\documentclass[11pt,reqno]{article}

\usepackage[top=2.5cm, bottom=2.5cm, left=3.2cm, right=3.2cm]{geometry}	

\usepackage{amsmath, amsfonts, bm, amssymb, mathtools, bm, commath, amsthm}
\usepackage{doi}
\usepackage{color}
\usepackage[dvipsnames]{xcolor}
\usepackage[font=small,labelfont=bf]{caption}
\usepackage{enumitem}
\usepackage{chngcntr}
\usepackage{bm} 
\usepackage{algorithmic,algorithm}
\usepackage{cancel}
\usepackage[maxbibnames=99,backend=bibtex]{biblatex}
\usepackage[capitalise]{cleveref}
\crefname{enumi}{Assumption}{Assumptions}

\usepackage{xspace}
\hypersetup{
	colorlinks   = true, 
	urlcolor     = {green!50!black}, 
	linkcolor    = {green!50!black}, 
	citecolor   = {green!50!black} 
}

\renewcommand{\cal}[1]{\mathcal{#1}}

\renewcommand{\r}{\mathbb{R}}

\newcommand{\n}{\mathbb{N}}

\newcommand\numberthis{\addtocounter{equation}{1}\tag{\theequation}}

\newtheorem{definition}{Definition}
\newtheorem{theorem}{Theorem}
\newtheorem{corollary}{Corollary}
\newtheorem{lemma}{Lemma}
\newtheorem{proposition}{Proposition}
\theoremstyle{definition}
\newtheorem{example}{Example}
\newtheorem{remark}{Remark}
\newcommand\Xspace{\ensuremath{\mathsf{X}}\xspace}
\newcommand\Pspace{\ensuremath{\mathcal{P}(\Xspace)}\xspace}
\newcommand\Mspace{\mathcal{M}(\Xspace)}
\newcommand\Mzspace{\mathcal{M}_{0}(\Xspace)}

\newcommand\Pdspace{\mathcal{P}_\lambda(\Xspace)}
\newcommand\BorelX{\mathcal{B}(\Xspace)}

\newcommand\CbVfunX{\mathbf{C}_{V}(\Xspace)}
\newcommand\CbVfunXa{\mathbf{C}_{V^\alpha}(\Xspace)}
\newcommand\BfunX{\mathbf{B}_1(\Xspace)}
\newcommand\ConefunX{\mathbf{C}(\Xspace)}
\newcommand\BVfunX{\mathbf{B}_V(\Xspace)}
\renewcommand\BVfunX{\mathbf{B}_V(\Xspace)}

\newcommand\Pas{$\mathbb{P}$\text{-a.s.}\xspace}

\newcommand\supnorm[1]{\left\vert#1\right\vert_{\infty}}
\newcommand\tvnorm[1]{\norm{#1}_{\rm tv}}
\newcommand\Vnormf[1]{\left\vert#1\right\vert_{V}}
\newcommand\Vnormfa[1]{\left\vert#1\right\vert_{V^\alpha}}
\renewcommand\Vnormfa[1]{\left\vert#1\right\vert_{V^\alpha}}
\newcommand\Vnorm[1]{\norm{#1}_{V}}
\renewcommand\Vnorm[1]{\norm{#1}_{V}}
\newcommand\Vnorma[1]{\norm{#1}_{V^\alpha}}
\renewcommand\Vnorma[1]{\norm{#1}_{V^\alpha}}

\newcommand\Id{\ensuremath{\textrm{Id}}\xspace}

\newcommand\Pder[3]{\partial_\pi P_{#1}(#2,#3)}
\newcommand\Pderorth[3]{\partial_\pi P_{#1}^\perp(#2,#3)}

\newcommand\rhodomainP[1]{\mathcal{D}_{#1}(P)}

\newcommand\Pderintr[3]{D_\pi P_{#1}(#2,#3) }

\newcommand\Pderk[4]{\partial_\pi P_{#1}^{#4}(#2,#3)}

\newcommand\neigh[2]{\mathcal{N}(#1,#2)}

\DeclareMathOperator{\ess}{ess}

\title{A calculus for Markov chain Monte Carlo:\\   studying approximations in algorithms}
\author{
	Rocco Caprio\footnotemark[2]
	\and
	Adam M. Johansen\footnotemark[2]
}

\begin{document}
	\maketitle

	\begin{abstract} 
		Markov chain Monte Carlo (MCMC) algorithms are based on the construction of a Markov chain with transition probabilities leaving invariant a probability distribution of interest. In this work, we look at these transition probabilities as functions of their invariant distributions, and we develop a notion of \textit{derivative in the invariant distribution} of a MCMC kernel. We build around this concept a set of tools that we refer to as \textit{Markov chain Monte Carlo Calculus}. This allows us to compare Markov chains with different invariant distributions within a suitable class via what we refer to as mean value inequalities. We explain how MCMC Calculus provides a natural framework to study algorithms using an approximation of an invariant distribution, and we illustrate this by using the tools developed to prove convergence of interacting and sequential MCMC algorithms. Finally, we discuss how similar ideas can be used in other frameworks.	
		\vspace{1.5mm}
		
	\end{abstract}

	\section{Introduction}
	\footnotetext[2]{Department of Statistics, University of Warwick. Email: \{rocco.caprio, a.m.johansen\}@warwick.ac.uk.}

	Markov chain Monte Carlo (MCMC) algorithms are based on the construction of a Markov chain, invariant with respect to a distribution of interest, on some Polish space $(\Xspace,\BorelX)$. The collection of probability measures on this space will be denoted \Pspace. The transition probabilities of such a chain are described by a \emph{Markov kernel} on $\Xspace\times \BorelX$: a function $P:\Xspace\times\BorelX\mapsto [0,1]$ such that (i) for each $A\in\BorelX$, $P(\cdot,A)$ is a non-negative measurable function on $\Xspace$ and (ii) for each $x\in\Xspace$, $P(x,\cdot)\in\Pspace$. $P$ acts on suitably integrable functions via $P(x,f):=Pf(x):=\int f(y)P(x,\dif y)$, and on probability measures via $\rho P(\dif x):=P(\rho,\dif x):=\int \rho(\dif y)P(y,\dif x)\in\Pspace$. Whenever a probability measure $\mu$ on $\Xspace$ satisfies $\mu= P(\mu,\cdot)$, we say that $\mu$ is an invariant distribution for $P$. A Markov kernel having an invariant distribution $\mu$ will be denoted $P_\mu$ whenever the emphasis is important. 
	
    When one can evaluate the density of a target measure $\mu$ pointwise (up to a possibly unknown normalizing constant) there are many mechanisms for specifying a kernel $P_\mu$. It often happens, however, that $\mu$ is intractable, or just very hard to target directly. In these situations, possible and popular remedies include simply using an approximation $\mu_n$ of $\mu$ and hence simulating a homogeneous Markov chain with transitions $P_{\mu_n}$, or considering an evolving sequence $\{\mu_n\}_{n\in\mathbb{N}}$ converging to $\mu$ in some sense, and simulating an inhomogeneous Markov chain $\{P_{\mu_n}\}_{n\in\mathbb{N}}$, perhaps in a tempering-like context. In these situations, one usually has control over how good the approximation is, and it might be of interest to ensure that the ``approximate'' chain $P_{\mu_n}$ moves similarly to the ``limiting'' chain   $P_\mu$, and perhaps to quantify how the ``degree of similarity'' depends on the quality of our approximation.
	
	Roughly speaking, the ultimate goal of this paper is to look at Markov kernels as functions of their invariant distribution, and to be able to say that $P_{\mu_n}$ and $P_\mu$  are similar whenever we can say both that $\mu$ and $\mu_n$ are similar and that the derivative in the invariant distribution of $P_\cdot$ at $\mu$ towards $\mu_n$ is bounded. In this paper, we shall make sense of this sentence, first by introducing a notion of a family of Markov kernels, where one has a mapping from the invariant distribution to the Markov kernel (\cref{sec:markovfamilies}), and then of derivative in the invariant distribution and of boundedness thereof (\cref{sec:derivative,sec:ftc}). This is accomplished by adapting some standard functional derivatives to the Markov kernel setting and applying them in the MCMC context. We shall then provide a ``Fundamental theorem of MCMC Calculus’’ to compare Markov chains with different invariant distribution within the same family. When $\mu,\nu$ have positive densities, this result will allow us to easily derive \emph{mean-value} type inequalities 
	\begin{equation*}
		\tvnorm{P_\mu(\rho,\cdot)-P_\nu(\rho,\cdot)} \leq M_{\rho}\tvnorm{\mu-\nu}
	\end{equation*}
	when $\rho$ also has a positive density, and when $\rho=\delta_x$,
	\begin{equation*}
		\tvnorm{P_\mu(x,\cdot)-P_\nu(x,\cdot)} \leq M_x \tvnorm{\mu-\nu}+ M_{\perp,x}|\mu(x)-\nu(x)|,
	\end{equation*}
	for Hastings type and Gibbs kernels, with explicit values for the ``Lipschitz constants'' $M_{x}, M_{\perp, x}$ and $M_{\rho}$ (\cref{cor:MHboundedder,cor:Gibbsboundedder}).
        
        These results are immediately useful when one is using an approximate Chain $P_{\mu_n}$ and wants it to move like $P_{\mu}$: it suffices to pick a Markov kernel with finite Lipschitz constants
        (which will mean that the kernel has a bounded derivative in the invariant distribution) or to minimize fluctuations around $P_\mu$. In \cref{sec:applications} we use our calculus approach to study the efficiency of approximation-based MCMC algorithms like Sequential MCMC or Interacting MCMC in terms of their asymptotic variance. In short, the mean value inequalities allow us to verify the \textit{diminishing adaptation}-type conditions one usually has to verify in adaptive algorithms. Our approach allow for a very natural study of these, allowing us to state conditions for convergence simply in terms of good enough approximation quality, boundedness of the derivative, and enough uniformity of ergodic constants along the approximation sequence, and we improve existent convergence results. In Section \ref{sec:discussion} we conclude with a discussion on the applications of the tools developed beyond those considered here.

	
	\section{Markov Chain Monte Carlo Calculus} \label{sec:mcmccalc}
	
	\subsection{Families of Markov kernels} \label{sec:markovfamilies}
	To compare Markov kernels with different invariant distributions, our main tool will be the derivative in the invariant distribution of a Markov kernel, which we introduce in the next section. This quantity will describe, in some sense, the variation of the transition probabilities of a Markov Chain when $\mu$ is perturbed. Certainly, there are many possible transition probabilities associated with each invariant distribution, so to accomplish this we first have to restrict the class of Markov kernels considered. We introduce the notion of \textit{family of Markov kernels}---which is essentially a collection of Markov kernels indexed by a convex subset of invariant distributions. In the following definitions, and throughout this work, we denote with $\lambda$ a reference $\sigma$-finite measure. Usually, that will correspond to Lebesgue measure whenever $\Xspace=\r^{d}$ or the counting measure if $\Xspace=\mathbb{Z}$. $\cal{P}(\Xspace)$ denotes the set of probability measures on $\Xspace$, and $\Pdspace$ the set of probability measures with continuous bounded densities w.r.t.~$\lambda$. When $\mu\in\Pdspace$, we often abuse notation and just write $\mu(x)$ for the value of $\dif \mu/\dif \lambda$ at $x\in\Xspace$.
	
	\begin{definition}[Family of Markov kernels] \label{def:markovfamily}
		Let $\mathcal{I}$ be a convex subset of $\Pdspace$. A family of Markov kernels $\{P_\star\}$
		is a set $\{P_\mu:\Xspace \times \BorelX\mapsto[0,1] \vert \mu\in\mathcal{I} \}$, where for all $ \mu \in \mathcal{I}$ one has $ \mu P_\mu = \mu$.
	\end{definition}
	
	According to our definition, if for some $\mu,\nu\in\Pdspace$, $P_\mu, P_\nu \in\{P_\star\}$, then defining the curve $\mu_t:=(1-t)\mu + t\nu$, as the map $t\in[0,1]\mapsto \mu_t \in \Pdspace$, interpolating $\mu$ and $\nu$, we have that $P_{\mu_t}\in \{ P_\star\}$ for any $t\in[0,1]$. $\mu_t$ and $P_{\mu_t}$ will play an important role. The following two examples describe some Markov kernel families.

	\begin{example}[The Hastings family] \label{ex:MHtypekernels}
	  Let $\mu\in\Pdspace$ and let $Q$ be a \textit{proposal} Markov kernel such that $Q(x,dy)=q(x,y)\lambda(dy)$ for all $x\in\Xspace$ and let $g$ be a \textit{balancing function} satisfying $g(x)=xg(1/x)$, $g(x)\leq 1$. For $x\in\Xspace$ and suitable test functions $f$, consider the Markov kernels of the form
     \begin{equation} \label{eq:MHtypekernels}
        P_\mu(x,f) := \int f(y)Q(x,\dif y)g(r_\mu(x,y)) + f(x)\left(1-\int Q(x,\dif y)g(r_\mu(x,y))\right) 
     \end{equation}   
     where 
     \begin{equation*}  
        r_\mu(x,y) :=
                \frac{\mu(y)q(y,x)}{\mu(x)q(x,y)} \quad \text{if} \quad \mu(x)q(x,y)>0,\quad \text{and} \quad 	r_\mu(x,y) :=1 \quad \text{otherwise}. 
     \end{equation*}
     These kernels have $\mu$ as invariant distribution and are indexed  by the convex set of all probability distributions $\mu$ with a density w.r.t. $\lambda$ on $\Xspace$, or any convex subset thereof. That is, they form proper Markov families as per \cref{def:markovfamily}. For generic $Q$ and $g$, we refer to the \textit{Hastings family}, as \cite{hastings1970mcmc} introduced kernels of this general form. When, for instance, $Q$ is a Gaussian Random Walk (RW) and $g(x)=\min(1,x)$ one obtain a \textit{RW Metropolis family} \cite{metropolis1953mcmc} over the same set of invariant distributions. If $Q$ is a Gaussian RW and $g(x)=x/(1+x)$, they form a \textit{RW Barker family} \cite{barker1965mcmc} etc. In general, our framework imposes no requirement that $Q$ be fixed across the family as described in this simple example and in general it can depend upon $\mu$. 
	\end{example}

	\begin{example}[The Gibbs family] \label{ex:Gibbsfamily}
	  Suppose $\mu$ is the probability distribution of a random vector $(U_1,U_2)$. Let $\mu_{1|2}$ and $\mu_{2|1}$ two Markov kernels on $\Xspace\times\BorelX$ representing the conditional distributions of $U_1$ given $U_2$ and $U_2$ given $U_1$, respectively. The Markov kernels on $\Xspace^2\times\mathcal{B}(\Xspace^2)$ given for all $x\in\Xspace$ and suitable test functions $f$ by
		\begin{equation} \label{eq:DSGibbskernel}
			P_\mu((x,x'),f):=\int \mu_{1|2}(x',\dif y)\mu_{2|1}(y,\dif y')f(y,y')
		\end{equation}
		are $\mu$-invariant. Kernels of the type \eqref{eq:DSGibbskernel} form a proper Markov family---the (two-stage) deterministic scan Gibbs family. 
	\end{example}
	
	We considered here the two-stage Gibbs family purely for ease for exposition and simplicity of notation; corresponding results can be obtained for much more general Gibbs samplers. Considering other Markov kernel families is possible. For instance, if in \cref{ex:MHtypekernels} we also allow $Q$ to depend on the gradient of $\mu$ as in the Metropolis-Adjusted Langevin Algorithm (MALA; \cite{roberts1996MALA}), one could obtain a MALA family (over the probability distributions with differentiable densities). If $\mu_{1|2}$, $\mu_{2|1}$ in \cref{ex:Gibbsfamily} instead of being conditional distributions are Markov kernels invariant for the same, one could obtain a Metropolis-within-Gibbs family etc. The requirements imposed upon a Markov family are not restrictive, and it is in principle also possible to consider a Markov family that is using a given type of Markov kernel if the invariant distribution belong to a certain set (e.g. MALA if the invariant density is differentiable), and another Markov kernel otherwise (RW Metropolis if it is not); however, the kernel differentiability properties we discuss next would be harder to verify.
	
	\subsection{The derivative in the invariant distribution} \label{sec:derivative}
	
	Let $\{P_\star\}$ be a Markov family indexed by a convex set of probability measures admitting densities w.r.t. some $\sigma$-finite measure $\lambda$. We will define the derivative with respect to the invariant distribution of $P_\cdot$ within that family at a point $\mu$ by first defining the derivative of the functional $P_\cdot(\rho,f):\mu\in\Pspace\mapsto P_\mu(\rho,f)\in\mathbb{R}$, for all $(\rho,f)$ in some suitable class. 
    In the below, we let $\Mzspace$ denote the set of $0$-mass finite signed measures on $\Xspace$. For $V:\Xspace\rightarrow [1,\infty)$, we define the $V$-norm of a function $f:\Xspace\mapsto \mathbb{R}$ and of a signed measure $\chi\in\Mspace$ respectively as 
	\begin{equation*}
		\Vnormf{f}:=\sup_{x\in \Xspace} \frac{\vert f(x) \vert}{V(x)} \quad \text{and} \quad \Vnorm{\chi}:=\sup_{f \in \BVfunX}\left\vert \int f\dif \chi\right\vert,
	\end{equation*}
        where $\BVfunX$ is the collection of functions with bounded $V$-norm.
	When $V=1$, the $V$-norm of a function and of a signed measure are just the supremum norm $\supnorm{\cdot}$ and the total variation norm $\tvnorm{\cdot}$, respectively. Throughout this section, we fix such a $V:\Xspace \mapsto [1,\infty)$, and consider $f \in \BVfunX$. The $V$ function directly influences how stringent the differentiability statement is, but also what we are able to say with the objects we introduce next, and its role will become clear when we introduce mean value inequalities later. 
	\begin{definition} \label{def:derivative}
		The derivative of $P_\cdot(\rho,f)$ in the invariant distribution at $\mu\in\Pdspace$, if it exists, is any linear functional $\Pder{\mu}{\rho}{f}[\cdot]:\Mzspace\mapsto \mathbb{R}$ such that for all $\nu\in\Pdspace$,
		\begin{equation} \label{eq:derivative}
			\frac{\dif }{\dif t} P_{\mu+t(\nu-\mu)}(\rho,f)\biggr|_{t=0} = \Pder{\mu}{\rho}{f}[\nu-\mu] 
		\end{equation}
	\end{definition}
	
		We can now leverage these quantities to define the derivative of transition probabilities.
	\begin{definition}
		The derivative of $P_\cdot(\rho,\cdot)$  in the invariant distribution at $\mu$ is the operator $\Pder{\mu}{\rho}{\cdot}[\cdot]:f\in \BVfunX \mapsto \Pder{\mu}{\rho}{f}[\cdot]$, where $\Pder{\mu}{\rho}{f}[\cdot]$ is defined in Definition \ref{def:derivative}.    
	\end{definition} 
	The differentiability of  $P_\cdot(\rho,\cdot)$ amounts to ask that $P_\cdot(\rho,f)$ is differentiable in the sense of Definition \ref{def:derivative} for all $f\in\BVfunX$. We denote with $\rhodomainP{\mu}$ the set of initial distributions $\rho$ for which $P_\cdot(\rho,\cdot)$ is differentiable at $\mu$. We are now in a position to define  the derivative of a Markov kernel itself.
	\begin{definition}
		The derivative of $P_\cdot$ in the invariant distribution at $\mu$ is the operator $\Pder{\mu}{\cdot}{\cdot}[\cdot]:(\rho,f)\in \rhodomainP{\mu}\times\BVfunX \mapsto \Pder{\mu}{\rho}{f}[\cdot]$, where $\Pder{\mu}{\rho}{f}[\cdot]$ is defined in \cref{def:derivative}.    
	\end{definition} 	
	We will sometimes refer to the derivative in the invariant measure of a Markov Kernel simply as ``kernel derivative''. As this quantity is defined indirectly for each point $(\rho,f)\in\rhodomainP{\mu}\times\BVfunX$, we sometimes also refer to  $\Pder{\mu}{\rho}{f}$ as a kernel derivative. 
	In all our examples, when $\rho\in\rhodomainP{\mu}$ has a density w.r.t. $\lambda$, the action of the functional $\Pder{\mu}{\rho}{f}$ will be expressible in integral form. In particular, for a measurable function abusively denoted as $\Pder{\mu}{\rho}{f}(\cdot):\Xspace\mapsto \mathbb{R}$ we will have
	\begin{align}
		\Pder{\mu}{\rho}{f}[\nu-\mu] = \int (\nu(y)-\mu(y))\Pder{\mu}{\rho}{f}(y)\lambda(\dif y) \quad \text{and}  \label{eq:derivativefirstvariation} \\
		\mu(\Pder{\mu}{\rho}{f}(\cdot))=0.  \label{eq:derivativecentering}
	\end{align}
	We sometimes distinguish the functional  $\Pder{\mu}{\rho}{f}[\cdot]$ from the function of its integral representation only via the brackets that follow. The context will always make obvious to which object we are referring to. The function $\Pder{\mu}{\rho}{f}(\cdot)$ will be referred to as density of the derivative in the invariant distribution. The density of $\Pder{\mu}{\rho}{f}$ corresponds to the $\cal{L}^2$-\emph{first variation} of $\mu\mapsto P_\mu(\rho,f)$, and its analogues appears for instance in \cite{cardaliaguet2019mastereqn} in the context of Mean Field Games and in Optimal Transport in \cite{santambrogio2015ot}.  Under some conditions on $\Xspace$ and $\Pspace$, these can be related to notions of derivatives in the Wasserstein geometry \cite{ambrosio2005gradientflows}, see also \cref{prop:FTC}. Here, \eqref{eq:derivativecentering} is a centering condition, and it guarantees that the integral representation is unique. Without, it would be unique only up to constants. 
	While for $\rho\in\Pdspace$ our $\Pder{\mu}{\rho}{f}$ will have a density, that will not be the case for atomic measures such as $\rho=\delta_x$.  In particular, we will more generally see that
	\begin{align} \label{eq:derivativesdec}
		\Pder{\mu}{\rho}{f}[\nu-\mu] = \int (\nu(y)-\mu(y))\Pder{\mu}{\rho}{f}(y)\lambda(\dif y) + \rho((\nu-\mu)\cdot\Pderorth{\mu}{\rho}{f})
	\end{align}
	with $	\mu(\Pder{\mu}{\rho}{f}(\cdot))=0$.
	Here, the function $\Pder{\mu}{\rho}{f}(\cdot)$ is the `density part' of the derivative, while $\Pderorth{\mu}{\rho}{f}(\cdot)$ plays the role of the `singular part'.
     Notice that if $\rho\in\Pdspace$ in the above, then we can just write \eqref{eq:derivativesdec} in the form of \eqref{eq:derivativefirstvariation} for an appropriate density function $\Pder{\mu}{\rho}{f}(\cdot)$.
	Let us come to the interpretation of these derivatives and differentials at $\mu$: they provide a map from the initial distribution to signed integral operators which
	capture the variation of the transitions associated with Markov kernels within a particularly family with small perturbations of the associated invariant measure, $\mu$. In other words, the derivative in the invariant measure quantifies how much the conditional expectation $P_{\mu}(\rho,\cdot)$ changes when the invariant distribution $\mu$ is perturbed by the $0$-mass measure $\chi=\nu-\mu$. Intuitively, if the derivative is \textit{large} in some sense, the Markov chain is not \textit{robust} to perturbations in the invariant measure. If we think of $\mu$ as an ideal measure we would like to target, this means that a chain targeting an approximation thereof would behave quite differently.  We shall make this more precise later. 
	The density of the kernel derivative $\Pder{\mu}{\rho}{f}(y)$---when it exists---can be heuristically interpreted as the change in the transition probabilities when $\mu$ is contaminated infinitesimally by a point mass $\delta_y$ for all $y\in\Xspace$. In fact, using \eqref{eq:derivative}--\eqref{eq:derivativecentering} we informally see that 
	\begin{equation*}
		\Pder{\mu}{\rho}{f}(y)=\Pder{\mu}{\rho}{f}[\delta_y-\mu] = \int \Pder{\mu}{\rho}{f}(y') (\delta_y-\mu)(\dif y') = \frac{\dif}{\dif t} P_{\mu+t(\delta_y-\mu)}(\rho,f)\bigg|_{t=0}.
	\end{equation*}
	This is similar to the role played by influence functions in robust statistics, which describe variations in estimators when the data distribution is contaminated---see \cite{huber2009robust}. 
 	In the end of the section we investigate the differentiability of Hastings and Gibbs families. First, we develop some basic properties to build up our intuition and understanding of kernel derivatives. 
	For $k\in\n$, define recursively the $k$-iterated kernel via $P^k(x,f) = \int P(x,\dif y)P^{k-1}(y,f)$ for all $x\in\Xspace$ and measurable $f$. $P^k$ represents the $k$-steps transition probabilities of the (homogeneous) Markov chain associated with $P$. We use the conventions $P^0_\mu(\cdot,f):=f$ and $P^0_\mu(\rho,\cdot):=\rho$. 
    The next proposition supplies a recursive formula for the derivative of $P^k$. These recursions are really an expression of the Chapman--Kolmogorov equations. 
	\begin{proposition} \label{prop:iteratedderivative} 
		For any $k\geq 2$, the derivative in the invariant distribution of a differentiable $k$-iterated Markov kernel satisfy the formul\ae
		\begin{equation*}
			\Pderk{\mu}{\rho}{f}{k}
			= \sum_{j=0}^{k-1} \Pder{\mu}{P^{k-j-1}_\mu(\rho,\cdot)}{P^j_\mu(\cdot,f)}.
		\end{equation*}
        provided $P^{j}_\mu(\rho,\cdot)\in\cal{D}_\mu(P)$ for all $j\leq k$.
	\end{proposition}
	\begin{proof} Let $\rho\in\Pdspace$ and $f\in\BVfunX$, we have, $with \mu_t = \mu + t(\nu - \mu)$: 
		\begin{align*}
			\frac{\dif }{\dif t}P^k_{\mu+t(\nu-\mu)}(\rho,f)\bigg|_{t=0} &=  \frac{\dif }{\dif t}\int P_{\mu+t(\nu-\mu)}(\rho,\dif z)P^{k-1}_{\mu_t}(z,f)\bigg|_{t=0}  \\
			&= \frac{\dif }{\dif t} P_{\mu+t(\nu-\mu)}(\rho,P^{k-1}_{\mu}(\cdot,f))\bigg|_{t=0} +  \frac{\dif }{\dif t} P_{\mu+t(\nu-\mu)}^{k-1}(P_\mu(\rho,\cdot),f)\bigg|_{t=0} \\
			&= \Pder{\mu}{\rho}{P^{k-1}_{\mu}(\cdot,f)}[\nu-\mu] + \Pderk{\mu}{P_\mu(\rho,\cdot)}{f}{k-1}[\nu-\mu]       
		\end{align*}
		therefore, exploiting this recursion yields,
		\begin{align*}
			\Pderk{\mu}{\rho}{f}{k} &= \Pder{\mu}{\rho}{P^{k-1}_{\mu}(\cdot,f)} + \Pderk{\mu}{P_\mu(\rho,\cdot)}{f}{k-1} 
			= \sum_{j=0}^{k-1} \Pder{\mu}{P^{k-j-1}_\mu(\rho,\cdot)}{P^j_\mu(\cdot,f)}.
		\end{align*}
	\end{proof}    
	Ergodicity can give us some information on how the derivative of iterated kernels behave when $k\rightarrow \infty$. A kernel derivative quantifies variations in transition probabilities $P^k_\mu(\rho,f)$ when the invariant changes is perturbed infinitesimally by $\chi:=\nu-\mu$. Under ergodicity for large $k$ we will have $P^k_\mu(\rho,f)\approx \mu(f)$, hence, intuitively, in the limit the kernel derivative ought to be $\chi(f)$. Define
	\begin{equation*}
		\neigh{\mu}{\nu}:=\{\mu_t\in\Pspace: \mu_t=(1-t)\mu+t\nu, t\in[0,1]\}.
	\end{equation*}
	This set is known as \textit{contamination neighbourhood} in robust statistics \cite{huber2009robust}, although it is not a neighbourhood in the sense of the total variation or the weak topology. For some starting distribution $\rho$, say that a Markov kernel is ergodic uniformly in $\neigh{\mu}{\nu}$ if $\sup_{t\in[0,1]}\Vnorm{P^k_{\mu_t}(\rho,\cdot)-\mu_t}\rightarrow 0$ as $k\to\infty$. We note that this is a definition of ergodicity uniform over a collection of invariant distributions and is quite distinct from the notion of uniform ergodicity of an individual Markov kernel.
	
	\begin{proposition}
		For all starting distributions $\rho\in\rhodomainP{\mu}$ for which the Markov kernel $P$ is ergodic uniformly in $\neigh{\mu}{\nu}$,
		\begin{equation*}
			\lim_{k\rightarrow\infty} \Pderk{\mu}{\rho}{f}{k}[\nu-\mu] = [\nu-\mu](f). 
		\end{equation*}
	\end{proposition}
	\begin{proof}
		We compute directly
		\begin{align*}
			\lim_{k\rightarrow\infty} \Pderk{\mu}{\rho}{f}{k}[\nu-\mu] &= \lim_{k\rightarrow\infty} \lim_{t\rightarrow0} \frac{P^k_{\mu+t(\nu-\mu)}(\rho,f)-P_\mu^k(\rho,f)}{t} \\
			&= \lim_{t\rightarrow0} \lim_{k\rightarrow\infty}  \frac{P^k_{\mu+t(\nu-\mu)}(\rho,f)-P_\mu^k(\rho,f)}{t} \\
			&=
			\lim_{t\rightarrow0}  \frac{(\mu+t(\nu-\mu))(f)-\mu(f)}{t} = [\nu-\mu](f),
		\end{align*}
		where the limit exchange is justified by the uniformity across the neighbourhood. 
	\end{proof}
	Hence, informally, its density in the limit is given by $\lim_k \Pderk{\mu}{\rho}{f}{k}(y)=f(y)-\mu(f)$.
	It is interesting to point out another connection between derivatives in the invariant distributions and ergodicity as follows. $\Pder{\mu}{\mu}{f}$ is understood as the infinitesimal variation of the transition probability $P_\mu(X,f)$ when the invariant measure is perturbed with a 0--mass measure, the starting point $X$ is chosen at random with the (non-perturbed) original distribution $\mu$. The density of the derivative takes a very natural expression in this case that does not require further computations. Denote $\mu_t := (1-t)\mu + t\nu$. 
	Since
	\begin{equation*}
		P_{\mu_t}(\mu_t,f) = \mu_t(f),
	\end{equation*}
	and differentiating this expression w.r.t. $t$ at $t=0$, provided that makes sense, gives 
	\begin{equation*}
		\frac{\dif }{\dif t} P_{\mu_t}(\mu,f)\Bigg|_{t=0} + P_{\mu_t}(\chi,f)\Bigg|_{t=0} + t\frac{\dif }{\dif t} P_{\mu_t}(\chi,f)\Bigg|_{t=0} = \chi(f).
	\end{equation*}
	Therefore, $\frac{\dif }{\dif t} P_{\mu_t}(\mu,f)\big|_{t=0}=\chi(f-P_\mu f)$, which via \eqref{eq:derivative} and \eqref{eq:derivativecentering} imply 
	\begin{equation} \label{eq:Gener=Deriv}
		\Pder{\mu}{\mu}{f}(y) = f(y)-P_\mu(y,f).
	\end{equation}
	This relationship can also be obtained heuristically using our interpretation of derivative in the invariant measure at a point $y$ as the variation in conditional expectation caused by a point mass $\delta_y$. We have then, up to constants,
	\begin{equation*}
		\Pder{\mu}{\mu}{f}(y) = \lim_{t\rightarrow 0} \frac{P_{\mu+t(\delta_y-\mu)}(\mu,f)-\mu(f)}{t},
	\end{equation*}
	but  $P_{\mu_t}(\mu,f)=\mu_t-tP_{\mu_t}(y,f)+tP_{\mu_t}(\mu,f)$ so that $P_{\mu_t}(\mu,f)=(\mu_t-tP_{\mu_t}(y,f))/(1-t)$. Hence, the right hand side of the expression above becomes
	\begin{align*}
		\lim_{t\rightarrow 0}\frac{\mu_t(f)-tP_{\mu_t}(y,f)-(1-t)\mu(f)}{(1-t)t} = \lim_{t\rightarrow 0}\frac{f(y) - P_\mu f(y)}{1-t} =  f(y)-P_\mu(y,f).
	\end{align*}
	This suggests an interesting connection with discrete-time generators and Foster--Lyapunov type arguments. One way to exploit this identity is simply by restating classical Markov Chains regularity conditions, very often expressed and verified in terms of drifts. The result below is a (partial) restatement of \cite[Theorem 13.0.1]{meyntweedie2012}.
	\begin{theorem}[Ergodic Theorem with kernel derivatives]
		Suppose that $\{X_n\}_{n\in\mathbb{N}}$ is an aperiodic, $\mu$-irreducible Markov Chain with transition probabilities $P_\mu$ and invariant distribution $\mu$. If there exists some petite set $C$, some $b<\infty$ and a non-negative finite function $f$ bounded on $C$ such that 
		\begin{equation*}
			-\Pder{\mu}{\mu}{f}(x)\leq -1 +b 1_C(x), \quad x\in\Xspace
		\end{equation*}
		whenever such kernel derivative exists, then for all $x\in\Xspace$, as $k\rightarrow\infty$,
			$\tvnorm{P_\mu^k(x,\cdot)-\mu}\rightarrow 0$.
	\end{theorem}
	Many other regularity conditions can be expressed in terms of drift conditions, and therefore in terms of $\Pder{\mu}{\mu}{f}$. 
	The main applications of kernel derivatives in this work will be however in the context of comparing Markov chains with different invariant distributions within the same family, and in particular in what we will refer to as mean value inequalities. First, for illustration, we compute the derivative in the invariant distribution for Markov kernels in the Hastings and Gibbs families of \cref{ex:MHtypekernels,ex:Gibbsfamily}. The next examples also illustrates the decomposition \eqref{eq:derivativesdec}.
	
	\begin{example}[Derivatives of Hastings kernels] \label{ex:MHtypekernelsdiff}
		Consider the Hastings family of \cref{ex:MHtypekernels}  with $x\mapsto g(x)$ twice differentiable, non-decreasing, and with $g'$ and $g''$ bounded. This class of kernel includes, for instance, the Barker form of the acceptance probability ($g(x)=x/(1+x)$), but not the  Metropolis-Hastings form  ($g(x)=\min(1,x)$). Furthermore, take the proposal distribution density to be bounded in both arguments, the indexed invariant distributions to have a positive densities with finite $V$-moments.
		Let $\mathcal{W}_{H}:=\{\rho\in\Pdspace:\rho/\mu^2 \text{ is bounded for all indexed $\mu$}\}$ denote the set of warm-start initial distributions, which essentially requires that $\rho$ has lighter tails than $\mu^2$. 
		\begin{proposition} \label{prop:MHderivatives} 
			The Hastings kernel is differentiable in the invariant distribution at $\mu$ for all $\rho\in\mathcal{W}_{H}\cup \{\delta_x\}_{x\in\Xspace}$. 
			For all $\rho\in\mathcal{W}_{H}$, $\Pder{\mu}{\rho}{\cdot}$ admits a density given for all $f\in\BVfunX$ by
			\begin{align*}  \label{eq:MHderivativedensity}
				\Pder{\mu}{\rho}{f}(y) =  
				\int(f(y)-f(z))&\frac{\rho(z)}{\mu(z)}g'(r_\mu(z,y))q(y,\dif z) \\
				&- \frac{\rho(y)}{\mu(y)^2}\int (f(z)-f(y))q(z,y)g'(r_\mu(y,z))\mu(\dif z) \numberthis.
			\end{align*}   
			$\Pder{\mu}{x}{\cdot}$ admits instead a density part and a singular part, given for all $f\in\BVfunX$ by
			\begin{align} \label{eq:MHderivativex}
				\Pder{\mu}{x}{f}(y) &= 
				(f(y)-f(x))\frac{g'(r_{\mu}(x,y))q(y,x)}{\mu(x)}\quad \text{and} \nonumber\\
				\Pderorth{\mu}{x}{f}(y)&= -\frac{1}{\mu(y)^2}\int (f(z)-f(y))q(z,y)g'(r_{\mu}(y,z))\mu(\dif z).
			\end{align}
		\end{proposition}
		\begin{proof}
			If we can interchange differentiation and integration, one has, with $\chi:=\nu-\mu$, 
			\begin{align*}
				\frac{\dif }{\dif t}P_{\mu+t(\nu-\mu)}(\rho ,f)\bigg|_{t=0} 
				&= \int \int \rho(\dif u) f(w)Q(u,\dif w)g'(r_\mu(u,w)) \frac{\dif }{\dif t} r_{\mu_t}(u,w)\bigg|_{t=0} \\
				&- \int \rho(\dif u)f(u)\int Q(u,\dif w)g'(r_\mu(u,w)) \frac{\dif }{\dif t} r_{\mu_t}(u,w)\bigg|_{t=0} \\
				&= \int \int (f(w)-f(u))\frac{q(w,u)g'(r_\mu(u,w))}{\mu(u)}\rho(\dif u)\chi(\dif w) \\
				&- \int \int(f(w)-f(u))  \frac{q(w,u)g'(r_\mu(u,w))}{\mu(u)^2}\mu(w)\rho(\dif u)\chi(\dif u)\\ 
				&= \int \left[\int (f(w)-f(u)) \frac{q(w,u)g'(r_\mu(u,w))}{\mu(u)}\rho(\dif u)\right] \chi(\dif w) \\
				&- \int \left[\int (f(w)-f(u)) \frac{q(w,u)g'(r_\mu(u,w))\mu(w)}{\mu(u)^2} \lambda(\dif w)\right]\chi(u) \rho(\dif u).
			\end{align*}         
			Hence, using the expressions above and \eqref{eq:derivative} one identifies \eqref{eq:MHderivativedensity} and \eqref{eq:MHderivativex}.  Appendix \ref{app:technicalcondboundder} shows that the the limit exchange performed is justified under our assumptions.
		\end{proof}
		It is possible to find kernel derivatives with other assumptions on the starting distributions, on the invariant distributions, the proposal and the balancing function by adapting the arguments used above and in the relevant appendixes; the only  difficulty is justifying the limit exchange. Our use of $V$ notationally suggests the use of a $V$ that is a suitable Lyapunov function for the different distributions indexed by the family---see \cref{sec:applications}.  The requirement that $g$ is differentiable however cannot be immediately relaxed. Although this means that the result above does \emph{not} guarantee differentiability of Metropolis-Hastings kernels, we will be able to obtain results for that class via an approximation argument later---and we will find that the Metropolis-Hastings kernel is actually an example of non-differentiable but Lipschitz continuous mapping.
	\end{example}
	
	In the example below, a probability distribution with a numerical subscript denotes the marginal in the indicated coordinate of the underlying probability distribution, we write $y=(y_1,y_2)$ for a point $y\in\Xspace^2$, $\lambda^2$ for $\lambda\otimes\lambda$.
	\begin{example}[Derivatives of Gibbs kernels] \label{ex:Gibbstypekernelsdiff} 
		Consider the MCMC kernels of \cref{ex:Gibbsfamily}, with $\mu_{1|2}$ and $\mu_{2|1}$ admitting positive densities $\lambda$-almost everywhere. 
		Define the set $\mathcal{W}_G:=\{\rho\in\Pdspace:\rho_2/ \mu_2 \text{ is bounded for all indexed $\mu$}\}$.
		\begin{proposition} \label{prop:Gibbsderivatives}
			The Gibbs kernel $P_\cdot(\rho,\cdot)$ is differentiable in the invariant distribution at $\mu$ for all $\rho\in\mathcal{W}_{G}\cap \{\delta_x\}_{x\in\Xspace}$. For all $\rho\in\cal{W}_G$, the derivative $\Pder{\mu}{\rho}{\cdot}$ admits an integral representation, with its density given for all $f\in\mathbf{B}_{V}(\Xspace^2)$ by
			\begin{align*}  \label{eq:DSGibbsderivative}
				\Pder{\mu}{\rho}{f}(y) &=  \int  f(y_1,w_2)\frac{\rho(y_2)}{\mu_2(y_2)}\mu_{2|1}(y_1,\dif w_2) 
				- \int f(w_1,w_2)\frac{\rho(y_2)}{\mu_2(y_2)}\mu_{1|2}(y_2,\dif w_1) \mu_{2|1}(w_1,\dif w_2) \\
				&+ \int  f(y_1,y_2) \frac{\rho(\dif w_2)}{\mu_1(y_1)} \mu_{1|2}(w_2,y_1) 
				- \int f(y_1,w_2)\frac{\rho(\dif u_2)}{\mu_1(y_1)} \mu_{1|2}(u_2,y_1) \mu_{2|1}(y_1,\dif w_2) 
				\numberthis.
			\end{align*}
				$\Pder{\mu}{x}{\cdot}$ admits instead a density part and a singular part, given for all $f\in\mathbf{B}_{V}(\Xspace^2)$ by 
			\begin{align} \label{eq:Gibbsderivativex}
				\Pder{\mu}{x}{f}(y)&= f(y_1,y_2)\frac{\mu_{1|2}(x_2,y_1)}{\mu_1(y_1)} + \int f(y_1,w_2)\bigg(\frac{\mu_{2|1}(y_1,\dif w_2)}{\mu_2(x_2)}-\frac{\mu_{1|2}(x_2,y_1)\mu_{2|1}(y_1,\dif w_2)}{\mu_1(y_1)}\bigg) \nonumber\\
				\Pderorth{\mu}{x}{f}(y)&= - \int f(w_1,w_2)\frac{\mu_{1|2}(y_2,\dif w_1)\mu_{2|1}(w_1,\dif w_2)}{\mu_2(y_2)}
			\end{align}      
		\end{proposition}
		\begin{proof}
			We notice that we can just write the full conditionals in terms of $\mu$ as $\mu_{1|2}(x',y)=\mu(y,x')/\int \mu(z,x')\lambda(\dif z)$ and similarly for $\mu_{2|1}$, hence 
			\begin{equation*}
				P_\mu((x_1,x_2),f)= \int  f(w_1,w_2) \frac{\mu(w_1,x_2)}{\int \mu(z,x_2)\lambda(\dif z)}\frac{\mu(w_1,w_2)}{\int \mu(w_1,z)\lambda(\dif z)}\lambda^2(\dif w).
			\end{equation*}
			Then, by exploiting the fact that the transition probabilities do not depend on the first coordinate of the starting point we compute
                        
			\begin{align*}
				\frac{\dif }{\dif t}P_{\mu+t(\nu-\mu)}(\rho ,f)\bigg|_{t=0} &= \int \chi(\dif w_1,u_2)\rho_2(\dif u_2)  \frac{f(w_1,w_2)}{\mu_2(u_2)} \mu_{2|1}(w_1,\dif w_2) \\
				&- \int \chi(\dif z,u_2)\rho_2(\dif u_2)  \frac{f(w_1,w_2)}{\mu_2(u_2)} \mu_{1|2}(u_2,\dif w_1) \mu_{2|1}(w_1,\dif w_2) \\
				&+ \int \chi(\dif w_1,\dif w_2)\rho_2(\dif u_2)  \frac{f(w_1,w_2)}{\mu_1(w_1)} \mu_{1|2}(u_2,\dif w_1) \\
				&- \int  \chi(\dif w_1,\dif z)\rho_2(\dif u_2)   \frac{f(w_1,w_2)}{\mu_1(w_1)}\mu_{1|2}(u_2,\dif w_1)\mu_{2|1}(w_1,\dif w_2).
			\end{align*}
			Using the expression above and \eqref{eq:derivative} one identifies \eqref{eq:DSGibbsderivative} and \eqref{eq:Gibbsderivativex}. The details on the limit exchange are in \cref{app:technicalcondboundder}. 
		\end{proof}     
		Again, the boundedness of $\rho_2/\mu_2$ is a warm-start type condition, and it is just required to formally justify the limit exchange.
	\end{example}

	Using the fact that the balancing function $g$ has to satisfy $g(x)=xg'(x)+g'(1/x)$, it is possible to check that when $\rho=\mu$, the derivative in the invariant measure of the Hastings kernels at $y\in\Xspace$ becomes (minus) the generator of the Markov chain, and the same is true for Gibbs kernels derivatives, as expected.
	
	\subsection{FTC Formula and bounded derivatives}\label{sec:ftc}
	
	We now state a simple formula that allow us to compare Markov Chains with different invariant distributions belonging to the same Markov family in terms of the kernel derivative. This is a Fundamental Theorem of Calculus-like formula, whose analogue appears in robust statistics \cite[Chapter 2]{huber2009robust} and mean field games \cite{cardaliaguet2019mastereqn}. Recall that $ \rhodomainP{\mu}$ denotes the set of initial distribution for which the Markov kernel is differentiable at $\mu$ and that $\mu_t:=(1-t)\mu+t\nu$.
	
	\begin{theorem}[Fundamental Theorem of MCMC Calculus Formula] \label{prop:FTC}
		It holds that
		\begin{equation} \label{eq:FTC}
			P_\mu(\rho,f)-P_\nu(\rho,f)= \int_0^1 \partial_{\pi}P_{\mu_t}(\rho,f)[\nu-\mu] \dif t 
		\end{equation} 
		for all $(\rho,f)\in \bigcap_{t=0}^1 \rhodomainP{\mu_t}\times \BVfunX$. Furthermore, if $\partial_{\pi}P_{\mu_t}(\rho,f)$ admits a continuously differentiable density and $\nu$ can be expressed via some pushforward $T$ from $\mu$ in that $\nu=T_* \mu$ we also have
		\begin{align*} \label{eq:FTCintrinsic}
			P_\mu(\rho ,f)-P_\nu(\rho,f) &= \int_0^1 \int \int^y_{T(y)} \Pderintr{\mu_t}{\rho}{f}(s) \dif s\mu(\dif y) \dif t \\
			&= \int_0^1 \int (T(y)-y) \int_0^1 \Pderintr{\mu_t}{\rho}{f}(s(T(y)-y)+y)\dif s\mu(\dif y) \dif t, \numberthis
		\end{align*}
		where $ \Pderintr{\mu}{\rho}{f}(y):=(\dif/\dif y)\Pder{\mu}{\rho}{f}(y)$.
	\end{theorem}
	\begin{proof}
		Let $(\rho,f)\in \bigcap_{t=0}^1 \rhodomainP{\mu_t}\times \BVfunX$. We then have
		\begin{equation*}
			P_\mu(\rho ,f)-P_\nu(\rho ,f) = \int_0^1 \frac{\dif }{\dif t} P_{\mu_t}(\rho ,f) \dif t = \int_0^1 \lim_{s\rightarrow0}\frac{P_{\mu_{t+s}}(\rho ,f)-P_{\mu_{t}}(\rho ,f)}{s}\dif t.
		\end{equation*}
		Now, since for all fixed $t$, $\mu_{t+s}=\mu_t+s(\mu_1-\mu_t)/(1-t)=\mu_t+s((\mu_1-t\mu_t)/(1-t)-\mu_t)$, and then using \eqref{eq:derivative} with $(\mu_1-\mu_t)/(1-t)\in \Mzspace$ being a 0--mass perturbation one can write under our hypotheses that for $t\in[0,1)$
		\begin{align*}
			P_\mu(\rho ,f)-P_\nu(\mu ,f) 
			= \int_0^1  \Pder{\mu_t}{\rho}{f}\left[\frac{\mu_1-\mu_t}{1-t}\right]\dif t  
			= \int_0^1  \Pder{\mu_t}{\rho}{f}[\nu-\mu]\dif t,
		\end{align*}
		with last equality following because $\mu_1-\mu_t=\nu-\mu_t=(1-t)(\nu-\mu)$. To prove \eqref{eq:FTCintrinsic}, use \eqref{eq:FTC} with $\mu_t:=(1-t)\mu+tT_*\mu$ to obtain
		\begin{align*}
			P_{\mu}(\rho,f)-P_{T_*\mu}(\rho,f) 
			&= \int_0^1 \int \Pder{\mu_t}{\rho}{f}(y)\dif (T_*\mu-\mu)(y)\dif t \\
			&= \int_0^1 \int \left[\Pder{\mu_t}{\rho}{f}(T(y))- \Pder{\mu_t}{\rho}{f}(y)\right]\mu(\dif y) \dif t.
		\end{align*}
		Since  $y\mapsto \Pder{\mu_t}{\rho}{f}(y)$ is continuously differentiable by assumption, an application of the Fundamental Theorem of Calculus completes the proof.
	\end{proof}
	
	Combining this with \cref{ex:MHtypekernelsdiff,ex:Gibbsderivatives} we have the following immediate corollary:
	\begin{corollary} \label{lemma:HastingsGibbsFTCisOK}
		For the Hastings and Gibbs families of  \cref{ex:MHtypekernelsdiff,ex:Gibbsderivatives}, \eqref{eq:FTC} holds for all $\cal{W}_H\cup \{\delta_x\}_{x\in\Xspace}\times\BVfunX$ and  $\cal{W}_G\cup \{\delta_x\}_{x\in\Xspace}\times\BVfunX$, respectively.
	\end{corollary}
	\begin{remark}
		The quantity $\Pderintr{\mu}{\rho}{f}$ can be related to rigorous notions of gradients in the Wasserstein geometry studied in \cite{ambrosio2005gradientflows}. It describes variations of the transition probabilities when $\mu$ is perturbed in the direction $\nu$ along the curves that are not simple interpolations like $\mu_t$ but rather \textit{transport maps}. In fact, one can substitute $T$ with $\Id-hT$ in \eqref{eq:FTCintrinsic} to obtain
		\begin{align*}
			P_{\mu}(\rho,f)-P_{T_*\mu}(\rho,f) 
			&=  h \int_0^1 \int T(y) \int_0^1  \Pderintr{\mu_t}{\rho}{f}(sh+y)\dif s\mu(\dif y) \dif t.
		\end{align*}
		Dividing by $h$ and letting $h\rightarrow 0$ we get
		\begin{equation*}
			\lim_{h\rightarrow 0}\frac{P_{(\Id-hT)_*\mu}(\rho,f)-P_\mu(\rho,f)}{h} = \int  \Pderintr{\mu}{\rho}{f}(y) T(y)\mu(\dif y).
		\end{equation*}
	\end{remark}
	
	We will employ \cref{prop:FTC} mainly as a way to obtain mean value type inequalities in the next section. However, formul\ae\xspace such as \eqref{eq:FTC} or \eqref{eq:FTCintrinsic} could  be used in general to obtain a possibly hard to guess expression for the difference in expectations given by two Markov kernels with different invariant distributions within the same family. We will often think of one of them as being an approximation of the other. 
	Notice that if $\rho=\nu$ for some $\nu$ indexed by the Markov family, one also obtains
	\begin{equation*} 
		P_\mu(\nu ,f)-\nu(f) = \int_0^1 \partial_{\pi}P_{\mu_t}(\nu,f)[\nu-\mu] \dif t
	\end{equation*} 
	an expression for the difference in expectation with respect to a Markov kernel applied to an initial distribution to that with respect to the distribution itself.

	\subsection{Bounded kernel derivatives and Mean-Value inequalities} \label{sec:mvi} 
	We want to argue that if the derivative in the invariant distribution at $\mu$ towards $\nu$ is bounded in some sense, then the Markov chains associated with kernels $P_\mu$ and $P_\nu$ move similarly whenever $\mu$ and $\nu$ are close.  
	We define the concept of bounded derivative in the invariant distribution accordingly. We aim to provide a statement similar to the fact that on $\mathbb{R}$ a differentiable function has a bounded derivative if and only if it is Lipschitz.  
	\begin{definition} \label{def:boundedderivative} 
		We say that a differentiable Markov kernel $P_\cdot(\rho,\cdot)$ has a bounded derivative at $\mu$ towards $\nu$ if there exist constants $M_{\rho},M_{\perp,\rho}<\infty$ such that
		\begin{equation} \label{eq:boundedderimplicationall}
			\Vnorm{P_{\mu}(\rho,\cdot)-P_{\nu}(\rho,\cdot)} \leq M_{\rho} \Vnorm{\mu-\nu}+M_{\perp,\rho}\rho(|\mu-\nu|).
		\end{equation}
	\end{definition}
	Noting that $d_\rho(\mu,\nu):=||\mu-\nu||_V + \rho(|\mu-\nu|)$ is itself a metric, this could be written as a standard statement of Lipschitz continuity with respect to that distance, with constant given by the larger of $M_{\rho}$ and $M_{\perp,\rho}$. However, the form we have here allow us to separate the role of two separate contributions and we believe that this refinement can be informative. With this in mind, we will sometimes refer to $M_{\rho},M_{\perp,\rho}$ as \textit{Lipschitz constants}, and to inequalities of the form \eqref{eq:boundedderimplicationall} as \textit{mean value inequalities}.  The next propositions show that \cref{def:boundedderivative} works as intended: we will see that if $P_\cdot(\rho,\cdot)$ is differentiable in $\mu_t=(1-t)\mu+t\nu$ for all $t\in[0,1]$ and a boundedness condition on the derivative holds, then $P_\cdot(\rho,\cdot)$ will have a bounded derivative in the invariant distribution at $\mu$ towards $\nu$. The boundedness depend on $V$, but we often silence this dependence because it is implicit in the context.  When $\rho\in\Pdspace$, this \cref{def:boundedderivative} is equivalent to the existence of a constant $M_\rho<\infty$ such that
	\begin{equation} \label{eq:boundedderimplication}
		\Vnorm{P_{\mu}(\rho,\cdot)-P_{\nu}(\rho,\cdot)} \leq M_{\rho} \Vnorm{\mu-\nu}.
	\end{equation}
	i.e. to the operator $P_\cdot(\rho,\cdot)$ from $\Pspace$ (via the invariant distribution) to  $\Pspace$ (the corresponding transition probabilities) being Lipschitz continuous w.r.t. $\Vnorm{\cdot}$. In fact, obviously if \eqref{eq:boundedderimplication} is true then \eqref{eq:boundedderimplicationall} holds with $M_{\perp,\rho}:=0$, and on the other hand when $\rho\in\Pdspace$ we can bound $M_{\perp,\rho}\rho(|\mu-\nu|)\leq M_{\perp,\rho}\sup(\rho/V)\Vnorm{\mu-\nu}$, and \eqref{eq:boundedderimplication} holds true with constant $M_{\rho}+M_{\perp,\rho}\sup(\rho/V)$.
	If $\rho=\delta_x$, \eqref{eq:boundedderimplicationall} becomes
	\begin{equation} \label{eq:boundedderimplicationx}
		\Vnorm{P_{\mu}(x,\cdot)-P_{\nu}(x,\cdot)} \leq M_x \Vnorm{\mu-\nu}+M_{\perp,x}|\mu(x)-\nu(x)|.
	\end{equation}
	This express the fact that as for many kernels employed in MCMC settings in particular, $P_\cdot(x,\cdot)$ typically depends on the values of the density of the invariant distribution at $x\in\Xspace$, one cannot expect $P_\mu(x,\cdot)$ to be close to $P_\nu(x,\cdot)$ unless $\mu(x)$ is close to $\nu(x)$, since the closedness of $\mu$ to $\nu$ in $V$-distance does not imply that their densities evaluated at individual points are close, at least without further assumptions.
		
	Notice that if a Markov kernel has a bounded derivative in the invariant distribution at $\mu$ towards $\nu$ and $\mu$ is close to $\nu$, the Markov Chains following $P_\mu(x,\cdot)$ and $P_\nu(x,\cdot)$ never move too differently starting from $x$, in that their $V$-distance is small. We can similarly consider other integral probability metrics and obtain related inequalities. Here we focused on the $V$-norm due its popularity within the MCMC literature arising from its close connection with Lyapunov functions for ergodicity \cite{R&R2004mcmc,meyntweedie2012}. In \cref{sec:applications} we exploit these connections.
	We now introduce a stronger form of boundedness. Let $\cal{Q}\subset \Pspace$. 
	\begin{definition} \label{def:unifboundedderivative}
		We say that a Markov kernel $P_\cdot(\rho,\cdot)$ has a $\mathcal{Q}$-uniformly bounded derivative at  $\mu$ towards $\nu$ if there  exist constants $M_1,M_2<\infty$ such that for all $\rho\in\mathcal{Q}$:
		\begin{equation*} 
			\Vnorm{P_{\mu}(\rho,\cdot)-P_{\nu}(\rho,\cdot)} \leq M_{1} \Vnorm{\mu-\nu}+M_{2}\rho(|\mu-\nu|).
		\end{equation*}
	\end{definition}
	If $\mathcal{Q}=\{\delta_x\}_{x\in\Xspace}$ the chains move alike uniformly in the space if their invariant distributions are close at all points. 
	Notice that a Markov kernel has an $\mathcal{Q}$-uniformly bounded derivative at $\mu$ towards $\nu$ if it has bounded derivative and the associated Lipschitz constants can be bounded uniformly over $\rho\in\mathcal{Q}$. See \cref{ex:Gibbsderivatives}, below, for such an example where $\mathcal{Q}$ is the set of warm-start distributions. 
	Finding conditions for $P_\cdot(\rho,\cdot)$ to have a bounded derivative in the invariant distribution when $\rho\in\Pdspace$ is rather immediate from \cref{prop:FTC}. Let $\ess_\rho\sup$ denote the essential supremum with respect to $\rho$.
	\begin{proposition} \label{prop:boundedderivativerho}
		Let the kernel derivative be in the form \eqref{eq:derivativesdec}. Assume that the Markov kernel $P_\cdot(\rho,\cdot)$ is differentiable at $\mu_t:=(1-t)\mu+t\nu$ for all $t\in[0,1]$ and that the constants
		\begin{align*}
			M_{\rho}:=&\sup_{f\in\BVfunX} \Vnormf{\int_0^1\Pder{\mu_t}{\rho}{f}(\cdot)\dif t},& M_{\perp,\rho}:=&\sup_{f\in\BVfunX}\ess_\rho\sup\bigg| \int_0^1\Pderorth{\mu_t}{\rho}{f}(\cdot)\dif t\bigg|,
		\end{align*}   
		are finite. Then,
        $P_\cdot(\rho,\cdot)$ has a bounded derivative at $\mu$ towards $\nu$ with Lipschitz constants $M_{\rho}$ and  $M_{\perp,\rho}$, hence the mean value inequality \eqref{eq:boundedderimplicationall} holds.
	\end{proposition}
	\begin{proof} 
		For all $\rho$ obeying the differentiability requirement and $f\in\BVfunX$, from \eqref{eq:FTC},
		\begin{align*}
			&|P_\mu(\rho,f)-P_\nu(\rho,f)| 
			= \left|\int_0^1 \int \Pder{\mu_t}{\rho}{f}(y) (\nu-\mu)(\dif y)+\rho((\nu-\mu)\cdot \Pderorth{\mu_t}{\rho}{f})\dif t\right| \\
			&\leq \bigg| \int \int_0^1 \Pder{\mu_t}{\rho}{f}(y)\dif t (\nu-\mu)(\dif y) \bigg| + \rho(|\nu-\mu|)\ess_\rho\sup\bigg|\int_0^1\Pderorth{\mu_t}{\rho}{f}(\cdot) \dif t\bigg|.
		\end{align*}        
		However, for any $f \in \BVfunX$, we have
		\begin{equation*}
			\left|\int\int_0^1\Pder{\mu_t}{\rho}{f}(y)\dif t (\nu-\mu)(\dif y) \right| \leq \Vnorm{\mu-\nu} \Vnormf{\int_0^1\Pder{\mu_t}{\rho}{f}(\cdot)\dif t},
		\end{equation*}
		and the claim follows upon taking the supremum over functions in $\BVfunX$.    
	\end{proof}
	Recall that if $\rho$ has a density, the singular part of the derivative is $0$ and $M_{\perp,\rho}=0$. In which case, the Proposition above says that the mapping $P_\cdot(\rho,\cdot):\mu\in(\Pspace,\Vnorm{\cdot})\mapsto P_\mu(\rho,\cdot)\in (\Pspace,\Vnorm{\cdot})$ is a Lipschitz continuous function of its invariant distribution with Lipschitz constant $M_{\rho}$.
	In the next example we apply this result to the Hastings and Gibbs kernels of \cref{ex:MHtypekernelsdiff,ex:Gibbstypekernelsdiff} via Propositions \ref{prop:MHderivatives}, \ref{prop:Gibbsderivatives} and Lemma \ref{lemma:HastingsGibbsFTCisOK}.
	\begin{example}[Bounded derivatives and Mean Value inequalities for Hastings kernels] \label{ex:MHboundedder}  
		\begin{corollary} \label{cor:MHboundedder}
			Let $\{P_\star\}$ be a Hastings family as per Example \ref{ex:MHtypekernelsdiff},  and suppose that the proposal density is bounded in both arguments. For all $\rho\in\cal{W}_H$, $P_\cdot(\rho,\cdot)$ has a bounded derivative at $\mu$ towards $\nu$ with Lipschitz constants $M_{\perp,\rho}=0$ and
				\begin{align} \label{eq:MHMVIlipconst}
				M_{\rho}:= \int_0^1 \int \bigg| (V(\cdot)+V(z))&\frac{\rho(z)}{\mu_t(z)}q(\cdot,z)g'(r_{\mu_t}(z,\cdot))\bigg|_V\lambda(\dif z) \\ &+ \int \bigg|(V(\cdot)+V(z)) {\mu_t}(z)q(z,\cdot)g'(r_{\mu_t}(\cdot,z))\frac{\rho(\cdot)}{\mu_t(\cdot)^2}\bigg|_V\lambda(\dif z) \dif t,\notag
			\end{align}   
			For all $x\in\Xspace$, $P_\cdot(x,\cdot)$ has a bounded derivative at $\mu$ towards $\nu$ with Lipschitz constants
			\begin{align} \label{eq:MHMVIxlipconst}
				M_x&:=\int_0^1 \Vnormf{(V(x)+V(\cdot))\frac{g'(r_{\mu_t}(x,\cdot))q(\cdot,x)}{\mu_t(x)}}\dif t, \quad \text{and} \\ M_{\perp,x}&:=\int_0^1 \int(V(x)+V(z))\frac{\mu_t(z)}{\mu_t(x)^2}q(z,x)|g'(r_{\mu_t}(x,z))|\lambda(\dif z)\dif t.
			\end{align}

		\end{corollary}
		\begin{proof}
			This follows from combining Propositions \ref{prop:boundedderivativerho} and \ref{prop:MHderivatives}.
		\end{proof}

		Although \cref{prop:MHderivatives} does not guarantee the differentiability of Metropolis-Hastings kernels (basically because its balancing function is not pointwise differentiable everywhere) in the invariant distribution, we can still obtain Lipschitz inequalities for such kernels indirectly by looking at Metropolis-Hastings kernels as limit of a sequence of differentiable ones: the function $g_j(x):=(x+\dots+x^j)/(1+x+\dots+x^j)$ is a proper balancing function, it yields a differentiable Markov family, and it satisfies $g_j(x)\rightarrow \min(1,x)=:g(x)$ for all $x>0$
		\cite{agrawal2023optimal}. Let $\{P_\star\}$ be now a Metropolis-Hastings family, with the same restrictions on the indexed invariants and $q$ as above.
		
		\begin{corollary}  \label{prop:MHrealMVI}
		For all $\rho\in\mathcal{W}_{H}$ the Lipschitz inequality \eqref{eq:boundedderimplicationall} holds with $M_{\perp,\rho}=0$ and
		\begin{equation} \label{eq:MHrealMVI}
			M_{\rho}:= \int_0^1 \int \Vnormf{(V(\cdot)+V(z)) \frac{\rho(z)}{\mu_t(z)}q(\cdot,z)}\lambda(\dif z) + \int \Vnormf{(V(\cdot)+V(z)) {\mu_t}(z)q(z,\cdot)\frac{\rho(\cdot)}{\mu_t(\cdot)^2}}\lambda(\dif z)\dif t.
		\end{equation}   
		The Lipschitz inequality \eqref{eq:boundedderimplicationx} holds with
		\begin{align} \label{eq:MHrealMVIx}
			M_x&:=\int_0^1 \Vnormf{(V(\cdot)+V(x))\frac{q(\cdot,x)}{\mu_t(x)}} \dif t \quad \text{and} \\
			M_{\perp,x}&:= \int_0^1\int (V(x)+V(z)) \frac{\mu_t(z)}{\mu_t(x)^2}q(z,x)1_{\{z:r_{\mu_t}(x,z)\leq 1\}}(z)\lambda(\dif z) \dif t.
		\end{align}

		\end{corollary}
		\begin{proof}
			Let $P_{\cdot,j}$ be the (differentiable) Hastings kernel employing the balancing function $g_j$. For $\rho\in\mathcal{W}_{2,H}$ proceeding as in \cref{prop:boundedderivativerho,cor:MHboundedder} we obtain $|P_{\mu,j}(\rho,f)-P_{\nu,j}(\rho,f)| \leq \Vnorm{\mu-\nu}M_{\rho,j} + |\mu(x)-\nu(x)|M_{\perp,\rho,j}$,
			where $M_{\rho,j}, M_{\perp,\rho,j}$ are as in the Proposition above, but with balancing function $g_j$. Noting then that $\sup_j g_j \leq 1$ we can take the limit $j\rightarrow \infty$ both sides. At the left hand side we obtain $|P_{\mu}(\rho,f)-P_{\nu}(\rho,f)|$ via bounded convergence. For the first term at the right hand side we can use $\sup_j |g'_j|\leq 1$ to bound $M_{\rho,j}\leq M_{\rho}$. For the second term at the right, upon taking the limit $j\rightarrow \infty$ at the left hand side we also use the fact that $g'_j(x)\rightarrow 1_{\{x\leq 1\}}(x)$ and the bounded convergence theorem to obtain the displayed Lipschitz constants.
		\end{proof}
		This shows that the Metropolis-Hastings kernel is an example of Lipschitz but non-differentiable mapping of its invariant distribution. 
		
		Roughly speaking, we have shown that under some conditions we can ensure that if $\mu$ and $\nu$ are close, two Hastings chain $P_\mu$ and $P_\nu$ move similarly with some proper choice of the initial distribution, the proposal distribution and the balancing function. 
		
	\end{example}

	\begin{example}[Bounded derivatives and Mean Value inequalities for Gibbs kernels] \label{ex:Gibbsderivatives} Let $\{P_\star\}$ be a deterministic scan Gibbs family as per Example \ref{ex:Gibbstypekernelsdiff}.
		\begin{corollary} \label{cor:Gibbsboundedder}
			For all $\rho\in\cal{W}_G$, $P_\cdot(\rho,\cdot)$ has a bounded derivative at $\mu$ towards $\nu$ with Lipschitz constants $M_{\perp,\rho}=0$ and 
			\begin{align}  \label{eq:GibbsMVIlipconst}     
				M_\rho &:=\int_0^1\int \bigg|V(\cdot_1,w_2)\frac{\rho_2(\cdot_2)}{\mu_{2,t}(\cdot_2)}\mu_{2|1,t}(\cdot_1,\dif w_2)\bigg|_V\\&+ \int \bigg|V(w_1,w_2)\frac{\rho_2(\cdot_2)}{\mu_{2,t}(\cdot_2)}\mu_{1|2,t}(\cdot_2,\dif w_1)\mu_{2|1,t}(w_1,\dif w_2)\bigg|_V  \nonumber \\ 
				& +\int \bigg|V(\cdot_1,\cdot_2)\frac{\rho_2(\dif u_2)}{\mu_{1,t}(\cdot_1)}\mu_{1|2,t}(u_2,\cdot_1)\bigg|_V + \int  \bigg|V(\cdot_1,w_2) \frac{\rho_2(\dif u_2)}{\mu_{1,t}(\cdot_1)} \mu_{1|2,t}(u_2,\cdot_1)\mu_{2|1,t}(\cdot_1,\dif w_2)\bigg|_V \dif t \nonumber
			\end{align}
			If additionally the full conditionals are bounded in both arguments, for all $x\in\Xspace$, $P_\cdot(x,\cdot)$ has a bounded derivative at $\mu$ towards $\nu$ with Lipschitz constants 
			\begin{equation} \label{eq:GibbsMVIlipconstx1}     
				M_{x}:=\int_0^1 \Vnormf{(V(\cdot)+V(x_1,x_2))\frac{\mu_{1|2,t}(x_2,\cdot_1)}{\mu_{1,t}(\cdot_1)}}+\int \Vnormf{V(\cdot_1,w_2)\mu_{2|1}(\cdot_1,\dif w_2)\frac{1}{\mu_{2,t}(x_2)}} \dif t,
			\end{equation}
			\begin{equation} \label{eq:GibbsMVIlipconstx2} 
				M_{\perp,x}:= \int  (V(x_1,x_2)+V(x_1,w_2)) \frac{\mu_{1|2,t}(x_2,\dif w_1)\mu_{2|1,t}(w_1,\dif w_2)}{\mu_{2,t}(x_2)}.
			\end{equation}
		\end{corollary}
		\begin{proof}
			This follows from combining Propositions \ref{prop:boundedderivativerho} and \ref{prop:Gibbsderivatives}. The claim on the finiteness of the Lipschitz constants follows from Bayes' formula.
		\end{proof}
	\end{example}
	
	For both Hastings and Gibbs kernels, to get $P_\mu$ and $P_\nu$ to move alike uniformly in the starting point $x\in\Xspace$ seems to be much harder on non-compact spaces.  
	
	\section{Applications} \label{sec:applications}
	\newcommand\YSeq{\{Y_{n,k};k\geq 0\}}
		
	The tools we developed allow a comparison between Markov chains of the same family with different invariant distributions. Usually we think of one of these distributions as an approximation of the other. Therefore, this framework might be appealing to study the use of approximations in Markov chain Monte Carlo. In this context, we compare a Markov Chain targeting a limiting distribution of interest, $\mu$, with a Markov Chain targeting an approximation thereof. 
	Using an MCMC kernel with an approximation is common and can occur in a number of contexts. For instance, when $\mu$ is intractable, or when $\mu$ is tractable, but rather inefficient to target directly, so one targets an approximation, say $\nu$, instead. Intuitively, we might be interested to ensure that the Markov chains $P_\mu$ and $P_{\nu}$ move similarly, and to control their distance. In fact, we show in the next section that (under suitable regularity conditions) this ensures that estimators associated with the two associated Markov chains achieve similar asymptotic variances, up to an additional variability due to the fluctuations of $\nu$ around $\mu$. 
	These approximations could be evolved or not as the Markov Chain runs, or the output of the resulting MCMC can be used to approximate the invariant distribution of another MCMC, giving rise to the framework for Sequential MCMC and Interacting MCMC we study in \cref{sec:smcmc,sec:imcmc}.

	\subsection{Efficiency of MCMC with approximations} \label{sec:AVconv}
	In this section we combine MCMC Calculus tools with results from \cite{fort2011convergence} and \cite{fort2014clt} to study the efficiency of Markov chains employing approximations in place of a limiting distribution of interest. In short, we will use our mean value inequalities to verify the \textit{diminishing adaptation} conditions \cite{R&R2007adaptive} that one requires to obtain theoretical guarantees for adaptive MCMC algorithms.
	
	Let $\{P_\star\}$ be a Markov Family. Consider a ``limiting'' Markov chain $\{X_k\}_{k\in\mathbb{N}}$ with transitions $X_{k+1}|X_{k}=x_{k} \sim P_\mu(x_{k},\cdot)$, targeting a distribution $\mu$ of interest, that might be intractable or inefficient to target directly. We consider two approaches to approximate $\{X_k\}_{k\in\mathbb{N}}$.  In the first we consider a situation in which one has access to a sequence $\{\mu_n\}_{n\in\n}$ of increasingly good approximations of $\mu$ and study the convergence of a sequence of approximating chains which make use of this sequence as its invariant measures; that is, it comprises a triangular array of Markov chains $\{Y_{n,k}\}_{k\leq n}$, where for each $k\leq n$, $Y_{n,k+1}|Y_{n,k}=y_{n,k} \sim P_{\mu_n}(y_{n,k},\cdot)$. 
	Since in our examples below $\mu_n$ comes from a previously run chain, we refer to this scheme as Sequential MCMC (sMCMC) --- see e.g. \cite{berzuini1997smcmc,golightly2006smcmc,finke2020smcmc,septier2009mcmc,li2019smcmc,li2023smcmc} among others. The second approach to approximate $\{X_k\}_{k\in\mathbb{N}}$ is to consider an (inhomogeneous) Markov chain $\{Z_k\}_{k\in\mathbb{N}}$ targeting at each $k$ a distribution $\mu_k$ that is refined at each step and that is getting closer to $\mu$ i.e. with transitions $Z_{k+1}|Z_{k}=z_{k} \sim P_{\mu_{k}}(z_{k},\cdot)$. We refer to this scheme as Interacting MCMC (iMCMC) --- see e.g.  \cite{andrieu2007nonlinearmcmc,atchade2010cautionarytale,brockwell2010simcmc,delmoraldoucet2010iMCMC,fort2011convergence} among others. 
	We are interested in studying the fluctuations of $n^{-1/2}\sum_{i=1}^n f(Y_{n,i})$ and  $n^{-1/2}\sum_{i=1}^n f(Z_i)$ around $\mu(f)$ for some test functions $f$ to be specified. Studying these fluctuations desirable as it provides information on the properties of sampled values viewed as estimators of expectations w.r.t. $\mu$ and characterizes the additional variability of estimators arising from using Markov kernels targeting (random) approximations of distributions rather than the distributions themselves. 
	Fluctuations for these algorithms have already been studied in some of the papers cited above. Here, we (sometimes greatly) extend the known theoretical guarantees and express the results somehow naturally in terms of boundedness of the derivatives in the invariant distributions and the quality of the approximation scheme $\mu_n\rightarrow \mu$.
	
	For both sMCMC and iMCMC, we assume  that $\mu_n\in\Pdspace$ for each $n\in\mathbb{N}$, 
    and we employ the following standard (if strong) ergodicity assumptions.
	\begin{enumerate}[label=A\arabic*]
		\item 
		\begin{enumerate}
			\item For all $n\geq 1$, $P_{\mu_n}$ are $\mu_n$-irreducible, aperiodic Markov kernels. Furthermore, there exists a function $V:\Xspace\mapsto [1,\infty)$ and constants $b<\infty$, $\lambda\in (0,1)$ such that for all $n \geq 1$:
			\begin{equation*}
				P_{\mu_n}V \leq \lambda V(x) + b1_C(x),
			\end{equation*}
			where $C:=\{x:V(x)\leq d\}$ for some $d\geq b/(2(1-\lambda))-1$. 
			\label{avar:hp:drift}
			\item There exist an integer $j\geq 1$, a positive constant $\kappa_n$ with $\inf_n \kappa_n>0$ and a probability measure $\upsilon_n$ such that for all $x\in C$, $P^j_{\mu_n}(x,A)\geq \kappa_n \upsilon_n(A)$ for all $A\in\BorelX$. 
			\label{avar:hp:minorization}
		\end{enumerate} 
		\label{avar:hp:unifergodicity}
		\item 
		There exist an integer $j>1$ such that $\sup_n \mu_n(V^j)<\infty$, $\sup_n \mathbb{E}(V^j(Z_n))<\infty$ (for $\{Z_k\}_{k\in\mathbb{N}}$) and $\sup_{n,k} \mathbb{E}(V^j(Y_{n,k}))<\infty$ (for $\{Y_{n,k}\}_{k\leq n}$).
		\label{avar:hp:Vgrowth}
	\end{enumerate}
    to which we add the following weak Feller property. Let $\CbVfunX$ denote the space of continuous functions with finite $V$-norm.
    \begin{enumerate}[label=B\arabic*]
        \item The Markov family $\{P_\star\}$ has the weak Feller property $f\in\CbVfunX\Rightarrow P_\mu f\in\CbVfunX$ for all indexed distributions $\mu$. \label{hp:Fellertypechain}
    \end{enumerate}
	\cref{hp:Fellertypechain} is a very mild property that, for a Hastings family, will be satisfied whenever the balancing function, the proposal and the invariants have suitable continuity and integrability properties. 
    \cref{avar:hp:unifergodicity} is a uniform (in the sequence of invariant distributions $\{\mu_n\}_{n\in\mathbb{N}}$) geometric drift and minorization assumption, similar to the conditions considered in \cite{andrieu2001sa} and in the adaptive MCMC literature (e.g. \cite{fort2011convergence,fort2014clt}). For the algorithms under examination, we can in fact look at the invariant distributions as an adaptation parameter living in the infinite dimensional space $\Pspace$. Using the techniques in \cite{fort2011convergence, fort2014clt} it is also possible to let $(\lambda,b)$ depend on $n$ in some judicious way, but we require their uniformity w.r.t. $n$ to allow the arguments which follow to focus upon the use of MCMC Calculus tools rather than technical details. Verifying \cref{avar:hp:unifergodicity} requires that $\mu_n$ does not behave too differently with $n$. For an illustration, let  $\{P_\star\}$ be the Random Walk Metropolis-Hastings family. If $q$ is jointly continuous and $\mu_n$ finite for all $n$, we can verify the minorization condition \cref{avar:hp:minorization} with $j=1$ for any compact set by noting that $P_{\mu_n}(x,A)\geq (\varepsilon/b_n) \mu_n(A)$, with $\varepsilon:=\inf_{x,y\in B}q(x,y)$, $b_n:=\sup_{x\in B}\mu_n(x)$, $B$ being any compact set with positive Lebesgue measure. If this holds, adapting \cite[Lemma 3.5]{jarner2000geometric} one can verify \cref{avar:hp:drift} whenever $V$ is continuous,
	\begin{equation*}
		\sup_n \limsup_{|x|\rightarrow \infty} \frac{P_{\mu_n}V(x)}{V(x)}<1 \quad \text{and} \quad \sup_n \sup_{x\in\Xspace} \frac{P_{\mu_n}V(x)}{V(x)}< \infty.
	\end{equation*}
	This can be verified by following the proof of  \cite[Theorem 3.2]{mengersen1996rates} with the Lyapunov function $V(x)=\exp(\gamma|x|)$ for Random Walk Metropolis-Hastings if all $\mu_n$ are uniformly log-concave in the tails i.e. when for all $n\geq 1$ there exist a $\gamma>0$ and $z\geq0$ such that 
	\begin{equation} \label{eq:uniflogconcave1}
		\log(\mu_n(x))-\log(\mu_n(y))\geq \gamma(y-x) \quad \text{for } y\geq x\geq z
	\end{equation}
	\begin{equation} \label{eq:uniflogconcave2}
		\log(\mu_n(x))-\log(\mu_n(y))\geq \gamma(x-y) \quad \text{for } y\leq x\leq -z.
	\end{equation}
	See \cref{avar:ex:smcmc} later for a concrete setting where we verify this. In the case where  $\{P_\star\}$ is an independent Metropolis-Hastings family $Q(x,\dif y)=Q(\dif y)$ and the proposal has heavier tail than any distribution in the approximation sequence in that $\inf_n \dif Q/\dif \mu_n(y) \geq \beta$ for some $\beta>0$, then the whole state space $\Xspace$ is small for every $P_{\mu_n}$ simultaneously, and the drift condition will hold with a bounded $V$.
	\cref{avar:hp:Vgrowth} is a growth condition which can actually always be verified with some drift function if \cref{avar:hp:unifergodicity} holds. In fact, via the drift condition \cref{avar:hp:drift} we have 
	\begin{equation*}
		\mathbb{E}(V(Z_n)) \leq \lambda\mathbb{E}(V(Z_{n-1})) + b \leq \lambda^n \mathbb{E}(V(Z_{0})) + b \sum^{n-1}_{i=0}\lambda^i
	\end{equation*}
	where we just iterated the first inequality in $n$, and upon taking the supremum in $n$ we obtain that $\sup_n \mathbb{E}(V(Z_n))<\infty$ and also $\sup_n \mu_n(V)<\infty$ via \cref{avar:lemma:geomerg} (i.e. \cref{avar:hp:Vgrowth} holds with $j=1$). On the other hand, for $j>1$, we can adapt the proof of \cite[Proposition 6.5]{rosenthal2023geometric} and define $\Tilde{V}:=V^{1/j}$, so that $\sup_n \mu_n(\Tilde{V}^j)=\sup_n \mu_n(V)<\infty$. By Jensen's inequality,
	\begin{align*}
		P_{\mu_n}\Tilde{V}\leq (P_{\mu_n}V)^{1/j} \leq (\lambda V+b1_C)^{1/j} \leq \Tilde{\lambda}\Tilde{V} + \Tilde{b}1_C
	\end{align*}
	with $\Tilde{\lambda}:=\lambda^{1/j}$ and $\Tilde{b}:=b^{1/j}$. Hence, $\Tilde{V}$ also satisfies \cref{avar:hp:unifergodicity} and arguing as above we may verify $\sup_n \mathbb{E}(V^j(Z_n))<\infty$ and $\sup_n \mu_n(V^j)<\infty$. For $\{Y_{n,k}\}_{k\leq n}$ it is similar.

	\subsubsection{Sequential MCMC} \label{sec:smcmc}
	For the Sequential MCMC sampling scheme  $\{Y_{n,k}\}_{k\leq n}$ we employ the following assumptions. Let $V$ satisfy \cref{avar:hp:unifergodicity}.
	
	\begin{enumerate}[label=C\arabic*]
		\item $\mu_n(x)\rightarrow\mu(x)$ \Pas for all $x\in\Xspace$. \label{avar:hp:ptwiseapprox}
		\item The Markov kernel $P_\cdot(x,\cdot)$ is differentiable in the invariant distribution at $\mu$, and has a $V$-bounded derivative at $\mu$ towards every $\mu_n$. \label{avar:hp:boundedder} 
		\item As $n\rightarrow 0$, for all $f$ in some functions class $\mathcal{G}(\Xspace)$ and for some variance functional $v$,   
		\begin{equation*}
			n^{-1/2}[\mu_n-\mu](f) \Rightarrow N(0,v(f)).
		\end{equation*}
		\label{avar:hp:Yapproxnormality}
	\end{enumerate}
	\vspace{-1.5pc}
	\cref{avar:hp:ptwiseapprox} is a requirement on the nature of the approximation scheme, which has to occur pointwise.
	\cref{avar:hp:boundedder} will ensure that the Markov Chains using $\mu$ and $\mu_n$ will in fact not move too differently when \cref{avar:hp:ptwiseapprox} holds, and it actually suffices that the boundedness is towards every $\mu_n$ for $n$ large enough. In fact, it is sufficient that a mean value inequality \eqref{eq:boundedderimplicationx} holds, and by \cref{prop:MHrealMVI} also the Metropolis-Hastings kernels satisfy that under the boundedness conditions stated therein.
	\cref{avar:hp:Yapproxnormality} is a mild assumption that says that if the employed approximation is random w.r.t. some $\mathbb{P}$, it also has a limiting Gaussian fluctuations around $\mu$, and it will be fulfilled by many approximation schemes used in practice. 
	
	Let $\sigma^2$ denote the asymptotic variance achieved by  $\{X_k\}_{k\in\mathbb{N}}$.
	\begin{theorem} \label{avar:thm:YCLT}
		Assume \cref{avar:hp:unifergodicity,avar:hp:ptwiseapprox,avar:hp:boundedder}. For all $f\in \CbVfunX$, 
		\begin{equation} \label{avar:eq:YrandomcenterCLT}
			n^{-1/2}\sum_{i=0}^n f(Y_{n,i})-\mu_n(f) \Rightarrow N(0,\sigma^2(f))
		\end{equation}
		and if \cref{avar:hp:Yapproxnormality} holds too, for $f\in\CbVfunX\cap\mathcal{G}(\Xspace)$,
		\begin{equation} \label{avar:eq:YdetcenterCLT}
			n^{-1/2}\sum_{i=0}^n f(Y_{n,i})-\mu(f) \Rightarrow N(0,\sigma^2(f)+v^2(f)).
		\end{equation}
	\end{theorem}
    \begin{proof}
        See \cref{app:smcmcproof}.
    \end{proof}
 
	\cref{avar:thm:YCLT} tells us in fact $\{Y_{n,k}\}_{k\leq n}$ possess a limiting Normal law around $\mu$ as $n\rightarrow\infty$, with the asymptotic variance equal to the asymptotic variance of an ideal MCMC scheme targeting $\mu$ \emph{plus} the variability due of the fluctuations of $\mu_n$.
	
	\begin{example}[Sequential MCMC for Feynman--Kac flow] \label{avar:ex:smcmc}
		\cite{berzuini1997smcmc,golightly2006smcmc, septier2009mcmc, li2023smcmc} among others consider the following algorithm, recently theoretically analyzed by \cite{finke2020smcmc}. Let $\{G^{(p)}\}_{p\in\mathbb{N}}$ be a sequence of non-negative functions, and $\{M^{(p)}\}_{p\in\mathbb{N}}$ be a sequence of Markov kernel, and $\eta^{(1)}$ an initial distribution.  
		\begin{algorithm}[H] \label{alg:mcmcpf}
			\caption{sMCMC for Feynman--Kac flow}\label{alg:admcmcpf}
			For $p=1$, 
			\begin{enumerate}
				\item Simulate $Y^{(1)}_{i+1}\sim P_{\eta^{(1)}}(Y^{(1)}_{i},\cdot)$ for $i=0,\dots,n$.
			\end{enumerate}
			Set $\eta^{(1)}_n:=n^{-1}\sum_{i=0}^n \delta_{Y^{(1)}_i}$. \\
			For $p\geq 2$,
			\begin{enumerate}
				\item Simulate $Y^{(p)}_{n,i+1}\sim P_{\Phi(\eta^{(p-1)}_n)}(Y^{(p)}_{n,i},\cdot)$ for $i=0,\dots,n$.
			\end{enumerate}
			Set $\eta^{(p)}_n:=n^{-1}\sum_{i=0}^n \delta_{Y^{(p)}_{n,i}}$. \\
		\end{algorithm}
		In the algorithm above, at each level $p$, the random variable $Y^{(p)}_0$ is sampled according to some initial distribution possibly depending on $p$. Here, we denoted $\Phi(\eta^{(p)})(\dif x):= \int\eta(\dif y)G^{(p)}(y)M^{(p)}(y,\dif x)/\eta(G^{(p)})$,  the so-called Boltzmann-Gibbs transformation associated to the Feynman--Kac model $\{G^{(p)},M^{(p)}\}_{p\in\mathbb{N}}$.
		
		This is a very flexible class of models, used perhaps most popularly in the filtering context within state space models, where one has an underlying unobserved Markovian state space process with some transition kernels $m^{(p)}$ at time $p$, and an observable process having likelihood $g^{(p)}$ at $p$. $\{G^{(p)},M^{(p)}\}_{p\in\mathbb{N}}$ can be identified with different relevant quantities related to $m^{(p)}$ and $g^{(p)}$, giving rise to different filters. The classic but usually inefficient choice $G^{(p)}=g^{(p)}$ and $M^{(p)}=m^{(p)}$ for all $p$ yields the Bootstrap particle filter, see \cite{delmoral2004fk, chopin2020smc}. Sampling from $\Phi(\eta^{(p)})$ with a particle filter amounts to being able to evaluate pointwise $G$ and sample i.i.d. from $M$. When one tries to implement more sophisticated filters, this might not be possible. In such case, a natural choice is to try to target $\Phi(\eta^{(p)})$ with a MCMC algorithm, giving rise to the algorithm above.  By targeting otherwise non implementable filters, this class of algorithms can outperform standard particle filters in some settings---see \cite{finke2020smcmc}, which gives the only CLT for $\sqrt{n}[\eta^{(p)}_n-\eta^{(p)}](f)$ of which we are aware. This CLT is shown to hold under conditions that are both restrictive and difficult to verify, and that have only been verified for independence samplers in finite spaces. We can extend their results using our theory. We prove below a CLT for bounded test functions for ease of exposition, but one can do much better using Theorem \ref{avar:thm:YCLT} above.
		We consider the following very mild assumption regarding the Feynman--Kac flow.  
		\begin{enumerate}[label=FK1]
			\item $G^{(p)}$ is positive and $M^{(p)}(x,\cdot)$ has a bounded density for all $x\in\Xspace$ and $p\in\mathbb{N}$. \label{avar:hp:FKflow} 
		\end{enumerate}
		Provided an LLN holds for the first Markov Chain $\{Y^{(1)}_k\}_{k\in\mathbb{N}}$, for each subsequent Chain $\{Y^{(p)}_{n,k}\}_{k\leq n},p>1$ we are in the setting studied for the process $\{Y_{n,k}\}_{k\leq n}$ earlier. 
		
		In particular, suppose that \cref{avar:hp:unifergodicity} holds. If follows that the Markov Chain $\{Y^{(1)}_{k}\}_{k\in\mathbb{N}},t>1$ satisfies a LLN and a CLT in that $\eta^{(1)}_n(f)\rightarrow\eta^{(1)}(f)$ \Pas and $\sqrt{n}[\eta^{(1)}_n-\eta^{(1)}](f)\Rightarrow N(0,\sigma_{\eta^{(1)}}(f))$ for all $f\in\BfunX$. By the decomposition
		\begin{equation} \label{avar:eq:FKdec}
			[\Phi(\eta^{(1)}_n)-\Phi(\eta^{(1)})](f) = [\eta^{(1)}_n-\eta^{(1)}](\bar{Q}^{(1)}(f-\Phi(\eta^{(1)}_n)(f)))
		\end{equation}
		with the integral operator $\bar{Q}^{(p)}(x,f):=G^{(p)}(x) M^{(p)}(x,f)/\eta^{(p)}(G^{(p)})$, under \cref{avar:hp:FKflow} one then also has $\Phi(\eta^{(1)}_n)(x)\rightarrow \Phi(\eta^{(1)})(x)$ \Pas for all $x\in\Xspace$,  $\Phi(\eta^{(1)}_n)(x)\rightarrow \Phi(\eta^{(1)})(x)$ for all $x\in\Xspace$ \Pas by separability and $\sqrt{n}[\Phi(\eta^{(1)}_n)-\Phi(\eta^{(1)})](f)\Rightarrow N(0,\sigma^2_{\eta^{(1)}}(\bar{Q}^{(1)}(f-\Phi(\eta^{(1)})(f)))$. If $P$ has a bounded derivative at $\Phi(\eta^{(1)})$ towards every $\Phi(\eta^{(1)}_n)$ then \cref{avar:thm:YCLT} shows
		\begin{equation*}
			[\eta^{(2)}_n-\eta^{(2)}](f)\Rightarrow N(0,\sigma^2_{\eta^{(2)}}(f-\eta^{(2)}(f))+\sigma^2_{\eta^{(1)}}(\bar{Q}^{(1)}(f-\eta^{(2)}(f))).
		\end{equation*}
		Repeating this argument allows us to establish the following result by induction. 
		\begin{theorem} \label{avar:thm:sMCMCFK}
			Let $p\in\mathbb{N}$ and $f\in\ConefunX$. Let $\{P_\star\}$ be a Markov family such that $\Phi(\eta^{(j)}),\Phi(\eta^{(j)}_n)$ for all $j\leq p$ belong to its index set, and for which \cref{avar:hp:unifergodicity,avar:hp:Vgrowth,hp:Fellertypechain} hold. If $P_\cdot$ has a bounded derivative at $\Phi(\eta^{(j)})$ towards every $\Phi(\eta^{(j)}_n)$ for all $j\leq p$ and if the Feynman--Kac model $\{G^{(p)},M^{(p)}\}_{p\in\mathbb{N}}$ satisfies \cref{avar:hp:FKflow} it holds
			\begin{equation*}
				[\eta^{(p)}_n-\eta^{(p)}](f)\Rightarrow N\left(0,\sum_{j=1}^p \sigma^2_{\eta^{(j)}}\left(\bar{Q}^{(j:p)}(f-\eta^{(j)}(f)\right)\right).
			\end{equation*}
			where we defined $\bar{Q}^{(j:p)}(f):=[\bar{Q}^{(j+1)} \circ \cdots \circ \bar{Q}^{(p)}](f)$ for $j<p$ and $Q^{(j:p)}(f)=\Id$ for $j=p$.
		\end{theorem}
	\end{example}
	
	For a concrete setting where this theorem is applicable, consider the following simple example, which can be easily extended. Let $p\in\mathbb{N}$ and let $\{P_\star\}$ be the Random Walk Metropolis-Hastings family. We consider a state-space model setting, a very flexible way to model quantities that evolve in time, and that are widely used in engineering, physics, quantitative finance and other fields, see e.g. \cite{chopin2020smc} for more details. In particular, we consider the model given by the latent process $W_{j+1}|W_{j}\sim N(w_{j+1};\varphi(w_j),1/2)$ and the observable $S_{j+1}|W_{j+1}\sim N(s_{j+1};w_{j+1},1/2)$, where $\varphi$ is a function bounded away from $-\infty$ and $+\infty$, say $-\bar{\varphi}\leq \varphi(x)\leq \bar{\varphi}$ for some finite and positive $\bar{\varphi}$. Let us consider a Bootstrap interpretation of the Feynman--Kac flow, by identifying the Markov kernels $M^{(p)}$ with $W$'s transitions, and the potential functions $G^{(p)}$ with the likelihoods of the observable. With this choice, it is easy to see that  \cref{avar:hp:FKflow} always hold. We now verify \cref{avar:hp:unifergodicity} with $V(x)=\exp(\gamma|x|)$ for any $0<\gamma\leq 2\bar{\varphi}$ by verifying \eqref{eq:uniflogconcave1} and \eqref{eq:uniflogconcave2} hence showing that $\{\Phi(\eta^{(j)}_n)\}_{n,j\in\n}$ are uniformly log-concave in the tails. Let $z:=2\bar{\varphi}$. If $y\geq x\geq z$, for $x,y\in\Xspace$, we compute 
	\begin{equation*}
		\frac{\Phi(\eta^{(j)}_n)(x)}{\Phi(\eta^{(j)}_n)(y)} = \frac{\sum_{i=1}^n e^{-(s_j-Y^{(j)}_i)^2}e^{-(x-\varphi(Y^{(j)}_i))^2}}{\sum_{i=1}^n e^{-(s_j-Y^{(j)}_i)^2}e^{-(y-\varphi(Y^{(j)}_i))^2}} = e^{-x^2+y^2} \frac{\sum_{i=1}^n e^{-(s_j-Y^{(j)}_i)^2+2x\varphi(Y^{(j)}_i) - \varphi(Y^{(j)}_i)^2}}{\sum_{i=1}^n e^{-(s_j-Y^{(j)}_i)^2+2y\varphi(Y^{(j)}_i) - \varphi(Y^{(j)}_i)^2}}.
	\end{equation*}
	Now, since  $e^{-(s_j-Y^{(j)}_i)^2+2x\varphi(Y^{(j)}_i) - \varphi(Y^{(j)}_i)^2} / e^{-(s_j-Y^{(j)}_i)^2+2y\varphi(Y^{(j)}_i) - \varphi(Y^{(j)}_i)^2} = e^{2\varphi(Y^{(j)}_i)(x-y)}$,
	\begin{equation*}
		\frac{\Phi(\eta^{(j)}_n)(x)}{\Phi(\eta^{(j)}_n)(y)} \geq \inf_i e^{-x^2+y^2}e^{2\varphi(Y^{(j)}_i)(x-y)} = \inf_i e^{(y-x)(y+x-2\varphi(Y^{(j)}_i))} \geq e^{\gamma (y-x)},
	\end{equation*}
	thus verifying \eqref{eq:uniflogconcave1}. \eqref{eq:uniflogconcave2} follows from analogous computations when considering the case $y\leq x\leq -z$. This verifies Assumption \ref{avar:hp:unifergodicity} and thus also shows that $\Phi(\eta^j_n)$ have $V$-moments for all $j,n\in\n$. Therefore, by \cref{prop:MHrealMVI}, $P_\cdot(x,\cdot)$ satisfies a mean value inequality at $\Phi(\eta^{(j)})$ towards every $\Phi(\eta^{(j)}_n)$, verifying the conditions of the theorem above.
	
	\subsubsection{Interacting MCMC} \label{sec:imcmc}
	For $\{Z_k\}_{k\in\mathbb{N}}$ we strengthen the assumptions to
	\begin{enumerate}[label=D\arabic*]
		\item $n^{-1/2}\sum_{k=1}^{n} \sup_x|\mu_{k}(x)-\mu_{k-1}(x)| \rightarrow 0$ and $n^{-1/2}\sum_{k=1}^{n} \Vnorm{\mu_k-\mu_{k-1}} \rightarrow 0$ \Pas. 
		\label{avar:hp:unifapprox}
		\item The Markov kernel $P_\cdot(x,\cdot)$ is differentiable in the invariant distribution at $\mu$, and has a $(V,\{\delta_x\}_{x\in\Xspace})$-uniformly bounded derivative at $\mu_n$ towards $\mu_{n-1}$ for all $n\geq 1$. 
		\label{avar:hp:unifboundedder}
		\item As $n\rightarrow 0$, for all $f$ in some functions class $\mathcal{G}(\Xspace)$ and for some variance functional $w$,
		\begin{equation*}
			n^{-1/2}\sum_{k=1}^{n}[\mu_k-\mu](f) \Rightarrow N(0,w(f)).
		\end{equation*}
		\label{avar:hp:Zapproxnormality}
	\end{enumerate}
	\vspace{-.75pc}
	The first condition \cref{avar:hp:unifapprox} strengthen \cref{avar:hp:ptwiseapprox} requiring that the approximation schemes is uniformly convergent on the space, and that the convergence occurs fast-enough. We believe that the second condition of \cref{avar:hp:unifapprox} will often follow from the first, just as \cref{avar:hp:ptwiseapprox} implies $\Vnorm{\mu_n-\mu}\rightarrow0$ \Pas (see \cref{app:lemma:ptwisetoV}) and this is quite immediate if $\Xspace$ is compact. \cref{avar:hp:unifboundedder} strengthens \cref{avar:hp:boundedder} to require that the ``Lipschitz constants'' of the differentiable Markov kernel $P$ are bounded in the starting point as per \cref{def:unifboundedderivative}.
	\cref{avar:hp:Zapproxnormality} is an asymptotic normality assumption concerning the approximation scheme that will often hold in the same situations as those in which \cref{avar:hp:Yapproxnormality} does---see \cref{avar:ex:imcmc}. 
	
	\begin{theorem} \label{avar:thm:ZCLT} 
		Let $\alpha\in(0,1/2)$. If \cref{avar:hp:unifergodicity,avar:hp:Vgrowth}, as well as \cref{avar:hp:unifapprox,avar:hp:unifboundedder}, hold, then for all $f\in \CbVfunXa$, 
		\begin{equation} \label{avar:eq:ZrandomcenterCLT}
			n^{-1/2}\sum_{i=1}^n f(Z_i)-\mu_{i-1}(f) \Rightarrow N(0,\sigma^2(f));
		\end{equation}
		and if \cref{avar:hp:Zapproxnormality} holds too, for $f\in\CbVfunXa\cap\mathcal{G}(\Xspace)$,
		\begin{equation} \label{avar:eq:ZdetcenterCLT}
			n^{-1/2}\sum_{i=1}^n f(Z_i)-\mu(f) \Rightarrow N(0,\sigma^2(f)+w^2(f)).
		\end{equation}
	\end{theorem}
	\begin{proof}
        See \cref{app:imcmcproof}.
    \end{proof}
	\cref{avar:thm:ZCLT} tells us that $\{Z_k\}_{k\in\mathbb{N}}$  also possess a limiting Normal law, with asymptotic variance equal to that of an ideal MCMC scheme targeting $\mu$ plus some additional variability due the approximation. This latter is often greater than the one we obtain with a Sequential MCMC scheme, as in the example below. 
	
	\begin{example}[Interacting MCMC for Feynman--Kac flow] \label{avar:ex:imcmc}
		\cite{andrieu2007nonlinearmcmc,atchade2010cautionarytale,brockwell2010simcmc,delmoraldoucet2010iMCMC,fort2011convergence} study the following type of algorithm.
		\begin{algorithm}[H] \label{alg:imcmc} 
			\caption{iMCMC for Feynman--Kac flow}
			For $n=1,...$
			\begin{enumerate}
				\item Simulate $Z^{(1)}_n\sim P_{\eta^{(1)}}(Z^{(1)}_{n-1},\cdot)$, set $\eta^{(1)}_n:= n^{-1}\sum_{i=1}^n \delta_{Z^{(1)}_i}$
				\item Simulate $Z^{(2)}_n\sim P_{\Phi(\eta^{(1)}_n)}(Z_{n-1}^{(2)},\cdot)$, set $\eta^{(2)}_n:= n^{-1}\sum_{i=1}^n \delta_{Z^{(2)}_i}$
				\item[] \centerline{\vdots}\vspace*{-0.5cm}
				\item[$p$.] Simulate $Z^{(p)}_n\sim P_{\Phi(\eta^{(p-1)}_n)}(Z_{n-1}^{(p)},\cdot)$, set $\eta^{(p)}_n:= n^{-1}\sum_{i=1}^n \delta_{Z^{(p)}_i}$
			\end{enumerate}
		\end{algorithm}
		At each level $p$, the random variable $Z^{(p)}_0$ is sampled according to some initial distribution possibly depending on $p$. Each $\{Z^{(p)}_k\}_{k\in\mathbb{N}},p>1$ is in the setting studied for the process $\{Z_k\}_{k\in\mathbb{N}}$. On the top of \cref{avar:hp:FKflow}, consider
		\begin{enumerate}[label=FK2]
			\item $y\mapsto M(x,y)$ is continuous on $\Xspace$ \Pas for all $x\in\Xspace$.
			\label{avar:hp:FKflow2}
		\end{enumerate}
		Under \cref{avar:hp:FKflow}, as argued in \cref{avar:ex:smcmc}, one has $\Phi(\eta^1_n)(x)\rightarrow \Phi(\eta^1)(x)$ for all $x\in\Xspace$ \Pas. If $\Xspace$ is compact and \cref{avar:hp:FKflow2} holds too, the convergence can be made uniform, verifying \cref{avar:hp:unifapprox}. To prove that \cref{avar:hp:Zapproxnormality} holds too, use \eqref{avar:eq:FKdec} to write 
		\begin{align*}
			&n^{-1/2}\sum_{k=1}^{n}[\Phi(\eta^{(1)}_k)-\Phi(\eta^{(1)})](f) \\
			&= \sum_{k=1}^{n}  \frac{n^{-1/2}}{k}  \sum_{j=1}^k [\bar{Q}^{(1)}(f-\Phi(\eta_k^{(1)})(f))(Z^{(1)}_j)-\eta^{(1)}(\bar{Q}^{(1)}(f-\Phi(\eta_k^{(1)})(f))] \\
			&= \int_0^1 \frac{n^{-1/2}}{s} \sum_{j=1}^{sk}[\bar{Q}^{(1)}(f-\Phi(\eta_k^{(1)})(f))(Z^{(1)}_j)-\eta^{(1)}(\bar{Q}^{(1)}(f-\Phi(\eta_k^{(1)})(f))]\dif s + o_{\mathbb{P}}(1)  \\
			&\Rightarrow \sigma_{\eta^{(1)}}(\bar{Q}^{(1)}(f-\Phi(\eta^{(1)})(f)) \int_0^1 \frac{B_s}{s} \dif s,
		\end{align*}
		where in the last line we used the fact that a (functional) CLT holds for $\{Z^{(1)}_n;n\geq 0\}$, and where $B$ is a standard Brownian Motion. The right hand side is a Gaussian random variable with variance $w(f):=2\sigma^2_{\eta^{(1)}}(\bar{Q}^{(1)}(f-\Phi(\eta^{(1)})(f))$. \cref{avar:thm:ZCLT} then applies to prove 
		\begin{theorem}
			Let $f\in\mathbf{C}_1(\Xspace)$. Assume that $\Xspace$ is compact, and that the Markov family $\{P_\star\}$ \cref{avar:hp:unifergodicity} holds and that $P$ has an uniformly bounded derivative at $\Phi(\eta^{(p)})$ towards every $\Phi(\eta^{(p)}_n)$ for all $p\leq 2$. Then, if the Feynman--Kac flow $(G,M)$ satisfies \cref{avar:hp:FKflow} and \cref{avar:hp:FKflow2} it holds
			\begin{equation}
				[\eta^{(2)}_n-\eta^{(2)}](f)\Rightarrow N(0,\sigma^2_{\eta^{(2)}}(f-\eta^{(2)}(f))+2\sigma^2_{\eta^{(1)}}(\bar{Q}^{(1)}(f-\eta^{(2)}(f))),
			\end{equation}
			where we defined the operator $\bar{Q}$ as in \cref{avar:thm:sMCMCFK}.
		\end{theorem}
		
		The case $p>2$ can be dealt with using \cref{avar:thm:ZCLT}, but it also requires complex multilevel expansion formul\ae---we refer to \cite{bercu2012iMCMCfluctuations}.
	\end{example}
	

	\section{Discussion} \label{sec:discussion} 
	
	We have developed herein a basic theory for methods allowing an easy and natural comparison between Markov Chains of the same family with different invariant distributions, by deriving analogues of the Fundamental Theorem of Calculus and Mean-Value Inequality for MCMC kernels viewed as functions of their invariant distributions. These tools allow us to study when Markov chains with different invariant distributions will move alike. This was done by `kernelizing' some concepts of functional derivatives that are very popular in other areas of science but not so much in the MCMC literature. Ultimately, deriving kernel derivatives and the mean value inequalities does not involve particularly complicated operations and we therefore think that the tools developed herein or related ideas could be helpful. Here, we used the mean value inequalities essentially to verify the \textit{diminishing adaptation} condition of some adaptive-like MCMC algorithms, where essentially one can easily study the evolution of the invariant distributions but not of the respective kernels. However, similar ideas can be used in other context. For instance, if it is the proposal distribution that changes between iterations, we can easily devise a `calculus with proposal distributions' and obtain similar mean value inequalities that can be used to prove diminishing adaptation for such algorithms. Furthermore, these mean value inequalities could be used in conjunction with \textit{perturbation bounds} (see e.g. \cite{rudolf2018perturb,medina2020perturb,alquier2016perturb}) to derive bounds for iterated kernels, and then infer distances between the relevant invariants.
	
	We believe these tools and ideas can be helpful also outside the adaptive MCMC and perturbation framework: recently, \cite{ascolani2024gibbs} derived the inequality
	\begin{equation} \label{eq:ascolanizanellamvi}
		\tvnorm{P_\mu(\rho,\cdot)-P_\nu(\rho,\cdot)} \leq 2 K M \tvnorm{\mu-\nu}
	\end{equation}
	for the deterministic scan Gibbs kernel, where $K$ represents the number of stages, for $\rho$ having a density satisfying the warm-start condition. \eqref{eq:ascolanizanellamvi} is then used to study mixing times for a Bayesian posterior distribution using information on the mixing time towards a Gaussian, which roughly speaking can be thought of as its large data limit in light of the Bernstein--von Mises theorem. \eqref{eq:ascolanizanellamvi} can be derived via Corollary \ref{cor:Gibbsboundedder} in the case $K=2$ using Bayes' formula (and the argument could be extended to the general case at the expense of introducing some slightly cumbersome notation). Analogous inequalities for the Metropolis-within-Gibbs kernel is derived in \cite{ascolani2024mwg}, which we expect can be computed alternatively using the tools developed herein. 
	
	\section*{Acknowledgements}
	
	\noindent RC was funded by the UK Engineering and Physical Sciences Research Council (EPSRC) via studentship 2585619 as part of grant number EP/W523793/1; AMJ acknowledges financial support from the Engineering and Physical Sciences Research Council (EPSRC; grants EP/R034710/1 and EP/T004134/1) and by United Kingdom Research and Innovation (UKRI) via grant EP/Y014650/1, as part of the ERC Synergy project OCEAN. 
	\noindent
	For the purpose of open access, the authors have applied a Creative Commons Attribution (CC BY) licence to any Author Accepted Manuscript version arising from this submission. Data sharing is not applicable to this article as no new data were created or analyzed in this study. 
	
	\printbibliography
	
	\appendix
	\section{Appendix to Section \ref{sec:mcmccalc}}	
	
	\subsection{Technical conditions in \cref{prop:MHderivatives}.}  \label{app:technicalcondboundder}

	Let $\rho\in\cal{W}_H$. To justify the interchange of differentiation and integration performed in \cref{prop:MHderivatives} one can simply show that for all $f\in\BVfunX$,  the function $\varphi_t(u,w):=|(\dif /\dif t)f(w)g(r_{\mu_t}(u,w))q(u,w)\rho(u)|$ is uniformly bounded in $t\in[0,1]$ by an $L^1(\lambda(\dif u)\otimes\lambda(\dif w))$ function (e.g. \cite[Theorem 2.27]{folland1999analysis}). Since $|f|\leq V$ and $g'$ is bounded, say by $c$,
	\begin{align*}
		&\varphi_t(u,w)=
		f(w)g'(r_{\mu_t}(u,w))\bigg|\frac{\chi(w)q(w,u)\mu_t(u)q(u,w)-\chi(u)q(u,w)\mu_t(w)q(w,u)}{(\mu_t(u)q(u,w))^2}\bigg|q(u,w)\rho(u) \\
		&\leq cV(w)\bigg|\frac{\chi(w)\mu_t(u)-\chi(u)\mu_t(w)}{\mu_t(u)^2}\bigg|q(w,u)\rho(u) 
		= cV(w)\bigg|\frac{\mu(w)\nu(u)-\mu(u)\nu(w)}{\mu_t(u)^2}\bigg|q(w,u)\rho(u) \\
		&\leq cV(w)\bigg(\frac{\mu(w)\nu(u)}{\mu(u)^2} + \frac{\mu(w)\nu(u)}{\nu(u)^2} + \frac{\nu(w)\mu(u)}{\mu(u)^2} + \frac{\nu(w)\mu(u)}{\nu(u)^2}\bigg)q(w,u)\rho(u)
	\end{align*}
	since $\sup_x \rho(x)/\min(\mu(x),\nu(x))^2<\infty$, the claim follows by the assumed $V$-integrability of the indexed invariants. For $\rho\in\{\delta_x\}_{x\in\Xspace}$ it is similar.
	
	\subsection{Technical conditions in \cref{prop:Gibbsderivatives}.}  
	Much as in the proof above, it suffices to show that for all $f\in\mathbf{B}_{V}(\Xspace^2)$, $\varphi_t(u,w):=|(\dif /\dif t)f(w_1,w_2)\mu_{1|2,t}(u_2,w_1)\mu_{2|1,t}(w_1,w_2)\rho_2(u_2)|$ is uniformly bounded in $t\in[0,1]$ by an $L^1(\lambda(\dif u)\otimes\lambda^2(\dif w))$ function. Notice that $\varphi_t(u,w)=\varphi_{1,t}(u,w)+\varphi_{2,t}(u,w)$, with
	\begin{align*}
		&\varphi_{1,t}(u,w)=|f(w_1,w_2)|
		\rho_2(u_2)\bigg|\frac{\chi(w_1,u_2)\mu_{2,t}(u_2)-\chi_2(u_2)\mu_t(w_1,u_2)}{\mu_{2,t}(u_2)^2}\bigg|\mu_{2|1,t}(w_1,w_2) \\
		&\leq V(w_1,w_2)\rho_2(u_2)\bigg|\frac{\chi(w_1,u_2)-\chi_2(u_2)\mu_{1|2,t}(u_2,w_1))}{\min(\mu_2(u_2),\nu_2(u_2))}\bigg|\mu_{2|1,t}(w_1,w_2) \\
		&\leq V(w_1,w_2)b(|\chi(w_1,u_2)|+|\chi_2(u_2)|\mu_{1|2,t}(u_2,w_1)))\mu_{2|1,t}(w_1,w_2) \\
		&\leq V(w_1,w_2)b(|\chi(w_1,u_2)|+|\chi_2(u_2)|(P_{1|2,0}(u_2,w_1))+P_{1|2,1}(u_2,w_1)))(P_{2|1,0}(w_1,w_2)+P_{2|1,1}(w_1,w_2))
		\intertext{where $b:=\sup_x \rho_2(x)/\min(\mu_2(x),\nu_2(x))<\infty$ by assumption. To bound $\psi_{2,t}(u,w)$ we use Bayes' formula to write}
		&\psi_{2,t}(u,w)=f(w_1,w_2)
		\rho_2(u_2)\bigg|\frac{\chi(w_1,w_2)\mu_{1,t}(w_1)-\chi_1(w_1)\mu_t(w_1,w_2)}{\mu_{1,t}(w_1)^2}\bigg|\mu_{1|2,t}(u_2,w_1) \\
		&=f(w_1,w_2) \rho_2(u_2)\bigg|\frac{\chi(w_1,w_2)\mu_{1,t}(w_1)-\chi_1(w_1)\mu_t(w_1,w_2)}{\mu_{1,t}(w_1)\mu_{2,t}(u_2)}\bigg|\mu_{2|1,t}(w_1,u_2)
	\end{align*}
	and then we proceed as for $\psi_{1,t}(u,w)$.
	
	\section{Appendix to Section \ref{sec:applications}}
	
	Before the proofs of \cref{avar:thm:YCLT,avar:thm:ZCLT} can be given, we need a number of auxiliary Lemmas and results. We use an approach based on the Poisson equations and resolvents. For a Markov kernel $P_\upsilon$ having $\upsilon$ as invariant measure, the Poisson equation is
	\begin{equation}
		\left\{
		\begin{array}{rl}
			(P_\upsilon-\Id)R_\upsilon =& \upsilon - \Id, \\
			\upsilon R_\upsilon =& 0,
		\end{array}
		\right.
	\end{equation}
	with $R_\upsilon$ being the resolvent operator defined by $ R_\upsilon f(x) := \sum_{k=0}^\infty [P_\upsilon^k(x,f)-\upsilon(f)].$ The following lemma is an immediate consequence of \cite[Lemma 2.3]{fort2011convergence}
	
	\begin{lemma} \label{avar:lemma:geomerg}
		If \cref{avar:hp:unifergodicity} holds, for all $n\in\mathbb{N}$ the Markov chains $\{P_{\mu_n}\}_{n\in\mathbb{N}}$ are $V$-geometrically ergodic in that for all $x\in\Xspace$, for some finite $C_n$ and $\beta_n\in (0,1)$,
		\begin{equation*}
			\Vnorm{P^k_{\mu_n}(x,\cdot)-\mu_n} \leq  V(x)C_n \beta_n^k.
		\end{equation*}
		Furthermore, $L:=\sup_n L_n<\infty$, where $L_n:=C_n \vee (1-\beta_n)^{-1}$.
	\end{lemma}
	
	Using the lemma above and the definition of the Poisson resolvent we also immediately obtain that for $f\in\BVfunX$, $|R_{\mu_n}f(x)|\leq \sum_{k=1}^\infty |P^k_{\mu_n}(x,f)-\mu_n(f)|\leq \Vnormf{f}V(x)C_n\sum_{k=0}^\infty\beta^k_n$ and hence the following. 
	\begin{lemma} \label{app:lemma:resolventbound}
		If \cref{avar:hp:unifergodicity} holds then, for all $n\geq 1$ and $\alpha\in(0,1]$, $|R_{\mu_n} f(x)|\leq V^\alpha(x)L^2\Vnormfa{f}$. In particular, $\sup_n \Vnorma{R_{\mu_n}f}<\infty$.
	\end{lemma}
	The role of $\alpha$ will become clear later.
	\begin{lemma} \label{app:lemma:ptwisetoV}
		If Assumptions \ref{avar:hp:unifergodicity}, \ref{avar:hp:Vgrowth} and \ref{avar:hp:ptwiseapprox} hold, then $\Vnorm{\mu_n-\mu}\rightarrow 0$. 
	\end{lemma}
	\begin{proof}
		Since $\Vnorm{\mu_n-\mu}=\int V(x)|\mu_n(x)-\mu(x)|\lambda(\dif x)$, by Scheff\'e's lemma, and the pointwise convergence given by \cref{avar:hp:ptwiseapprox}, it suffices to prove that $\mu_n(V)\rightarrow \mu(V)$. The convergence of $V$-moments is then verified by uniform integrability (e.g. \cite[Theorem 3.5]{billingsley1999cpm}) via \cref{avar:hp:Vgrowth}.
	\end{proof}

	The following Theorem gives the Law of Large Numbers we will need to prove the Central Limit Theorems for sMCMC. The proof is similar to \cite[Theorem 2.7]{fort2011convergence} but in a sequential rather an interacting setting. This result extends the LLN of \cite{finke2020smcmc} under much less stringent conditions, although it does not supply non-asymptotic estimates. 
	\begin{theorem} \label{app:thm:Yllns}
		Let $F:\Xspace\times\Pspace\mapsto \mathbb{R}^d$ be a measurable function satisfying $\sup_n \Vnormf{F(\cdot,\mu_n)}<\infty$ and assume $\lim_n \int F(x,\mu_n)\mu_n(\dif x)$ exists and that \cref{avar:hp:unifergodicity,avar:hp:Vgrowth} hold. Then, in probability,
		\begin{equation} \label{app:eq:Ylln}
			\lim_n n^{-1} \sum_{i=1}^n F(Y_{n,i},\mu_n) = \lim_n \int F(x,\mu_n)\mu_n(\dif x).
		\end{equation} 
		If \cref{avar:hp:ptwiseapprox} also holds and $F(\cdot,\mu):=\lim_n F(\cdot,\mu_n)$ exists, the latter limit is equal to $\int F(x,\mu)\mu(\dif x).$    
	\end{theorem}
	\begin{proof}
		Write $F_n:=F(\cdot,\mu_n), R_n:=R_{\mu_n}, P_n:=P_{\mu_n}$. We first use the Poisson equation to obtain the following decomposition
		\begin{equation*}
			n^{-1} \sum_{i=1}^n F_n(Y_{n,i}) - \mu_n(F_n) = M_n + r_n
		\end{equation*}
		where $M_n:= n^{-1} \sum_{i=1}^n R_n F_n (Y_{n,i}) - P_n R_n F_n (Y_{n,i-1}); r_n:=n^{-1}[P_nR_nF_n(Y_{n,0})-P_nR_nF_n(Y_{n,n})]$. The result follows if we show that both $M_n,r_n$ converge to 0 in probability.  Let $U_{n,i}:=n^{-1}[R_n F_n (Y_{n,i}) - P_n R_n F_n (Y_{n,i-1})]$. $\{U_{n,i};i\leq n\}$ is a triangular array of martingale difference sequences with respect to the triangular array of $\sigma$-fields $\{\mathcal{F}_{n,i}\}$, where $\mathcal{F}_{n,i}:=\bigvee_{m \leq n} \sigma(Y_{m,j}; j \leq m) \vee \sigma(Y_{n,j}; j \leq i)$. $M_n\rightarrow_{\mathbb{P}} 0$ then follows by \cite[Lemma A.1] {douc2007limit} if we establish the asymptotic negligibility of individual increments, i.e that for all $\epsilon>0$,
		\begin{equation*}
			\sum_{i=1}^n \mathbb{E}[|U_{n,i}|1_{|U_{n,i}|\geq \epsilon}|\mathcal{F}_{n,i-1}] \rightarrow_{\mathbb{P}} 0.
		\end{equation*}
		For any $\gamma>1$ and  $\epsilon>0$ we compute
		\begin{align*}
			\mathbb{E}[|U_{n,i}|1_{|U_{n,i}|\geq \epsilon}|\mathcal{F}_{n,i-1}] &\leq \epsilon^{-\gamma+1}\mathbb{E}[|U_{n,i}|^\gamma|\mathcal{F}_{n,i-1}]\\ &\leq \epsilon(n\epsilon)^{-\gamma}\mathbb{E}[|R_n F_n (Y_{n,i}) - P_n R_n F_n (Y_{n,i-1})|^\gamma|\mathcal{F}_{n,i-1}] \\
			&\leq  \epsilon(n\epsilon)^{-\gamma} 2^{\gamma-1}\mathbb{E}[|R_n F_n (Y_{n,i})|^\gamma + |P_n R_n F_n (Y_{n,i-1})|^\gamma|\mathcal{F}_{n,i-1}] \\
			&\leq  \epsilon(n\epsilon)^{-\gamma} 2^{\gamma} L^{2\gamma} \Vnormf{F_n(\cdot)}^\gamma P_n V^\gamma(Y_{n,i})\\ &\leq \epsilon(n\epsilon)^{-\gamma} 2^{\gamma}L^{2\gamma} \sup_n \Vnormf{F_n(\cdot)}^\gamma P_n V^\gamma(Y_{n,i})
		\end{align*}
		where in the third line we used the $C_p$ inequality, and in the fourth \cref{app:lemma:resolventbound}. \cref{avar:hp:Vgrowth} guarantees that, for some $\gamma > 1$, $P_n V^\gamma(Y_{n,i})$ is uniformly bounded in both $i,n$ and $M_n\rightarrow_{\mathbb{P}} 0$ follows. To prove $r_n\rightarrow_{\mathbb{P}} 0$ we use again \cref{app:lemma:resolventbound} to estimate
		\begin{align*}
			|r_n|\leq n^{-1}[|P_nR_nF_n(Y_{n,0})|+|P_nR_nF_n(Y_{n,n})|] \leq n^{-1}\sup_n \Vnormf{F_n(\cdot)} L^2(P_nV(Y_{n,0})+P_nV(Y_{n,n})).
		\end{align*}
		By \cref{avar:lemma:geomerg}, $P_n V(Y_{n,n}), P_nV(Y_{n,0})$ are both bounded in $n$, and \eqref{app:eq:Ylln} is proved. Now let $F:=F(\cdot,\mu)$, by the triangle inequality,
		\begin{equation*}
			|\mu_n(F_n)-\mu(F)|\leq |\mu_n(F_n)-\mu(F_n)|+|\mu(F_n)-\mu(F)|\leq \Vnorm{\mu_n-\mu}\Vnormf{F_n} +|\mu(F_n-F)|.
		\end{equation*}
		Since $\sup_n \Vnormf{F_n}<\infty$, under \cref{avar:hp:ptwiseapprox} the first term vanishes. Furthermore, we also have $|F_n-F|\leq 2V$, and the fact $\mu(V)<\infty$ guarantees that the second term also vanishes by the Dominated Convergence Theorem. It follows that   
		$\lim_n \int F(x,\mu_n)\mu_n(\dif x) = \int F(x,\mu)\mu(\dif x)$.   
	\end{proof}
	
	The next result is instead a Law of Large Numbers for iMCMC. This follows straight from \cite[Theorem 2.7]{fort2011convergence} together with a mean value inequality.
	
	\begin{theorem} \label{app:thm:Zllns}
		Assume that Assumptions \ref{avar:hp:unifergodicity}, \ref{avar:hp:Vgrowth}, \ref{avar:hp:unifapprox} and \ref{avar:hp:unifboundedder} hold. Let $F:\Xspace\times\Pspace\mapsto \mathbb{R}^d$ be a measurable function satisfying  $\sup_n \Vnorm{F(\cdot,\mu_n)}<\infty$, $\sum_{i=1}^\infty i^{-1}\sup_x \Vnormf{F(\cdot,\mu_i)-F(\cdot,\mu_{i-1})}<\infty$. Then, if $\lim_n \int F(x,\mu_n)\mu_n(\dif x)$ exists, in probability,
		\begin{equation} \label{app:eq:Zlln}
			\lim_n n^{-1} \sum_{i=1}^n F(Z_i,\mu_i) = \lim_n \int F(x,\mu_n)\mu_n(\dif x). 
		\end{equation} 
		If $F(\cdot,\mu):=\lim_n F(\cdot,\mu_n)$ exists, the latter limit is equal to $\int F(x,\mu)\mu(\dif x).$    
	\end{theorem}
	
	\begin{proof}
		\eqref{app:eq:Zlln} follows from \cite[Theorem 2.7]{fort2011convergence}. By \cref{avar:lemma:geomerg}, \cref{avar:hp:Vgrowth} and the implied fact that $\sup_n V(Z_n)$ is \Pas-finite, their assumptions their assumptions A3, A5 are all verified. To verify their A4 we apply a mean value inequality and the assumed uniformly bounded derivative in the invariant distribution \cref{avar:hp:unifboundedder} to write, for some $M_1,M_2<\infty$,
		\begin{align*}
			\sum_{i=1}^\infty i^{-1}(L_{\mu_i}\wedge L_{\mu_{i-1}})^6&\sup_x \Vnorm{P_{\mu_i}(x,\cdot)-P_{i-1}(x,\cdot)} V(Z_i)  \\
			&\leq  \sum_{i=1}^\infty i^{-1}(L_{\mu_i}\wedge L_{\mu_{i-1}})^6[M_1\Vnorm{\mu_i-\mu_{i-1}}+M_2\sup_x|\mu_i(x)-\mu_{i-1}(x)|]V(Z_i)  
		\end{align*}
		and the last expression is finite by $\sup_n L_{\mu_n}<\infty$, \cref{avar:hp:unifapprox} and the fact that $\sup_n V(Z_n)$ is \Pas-finite by \cref{avar:lemma:geomerg}, verifying the conditions in  \cite[Theorem 2.7]{fort2011convergence} and proving the claim. The fact that under \cref{avar:hp:ptwiseapprox} or \cref{avar:hp:unifapprox}  $\lim_n \int F(x,\mu_n)\mu_n(\dif x) = \int F(x,\mu)\mu(\dif x)$ follows from the same arguments used in \cref{app:thm:Yllns}.     
	\end{proof}
	In the proof below, for $\nu,\mu\in\Pspace$, we employ the operator identity
	\begin{equation} \label{avar:eq:PoiResolventkernelID}
		R_\nu-R_\mu = R_\mu(P_\nu-P_\mu)R_\nu + (\mu - \nu)R_\nu.
	\end{equation}
	The equation above is proved in \cite[proof of Proposition 3.1]{bercu2012iMCMCfluctuations}---we repeat the argument here for completeness. In operator sense, $P_\mu R_\mu = R_\mu P_\mu$ 
	Therefore
	it holds that $R_\mu(P_\mu-\Id)=(P_\mu-\Id)R_\mu=\mu-\Id$ by the Poisson equation, but then
	\begin{equation*}
		R_\mu(P_\mu-\Id)R_\nu = (\mu-\Id)R_\nu
	\end{equation*}
	and from the other side, again from an application of the Poisson identities, and since $R_\mu \nu(f)=0$,
	\begin{equation*}
		R_\mu(P_\nu-\Id)R_\nu = R_\mu(\nu-\Id) = -R_\mu
	\end{equation*}
	It follows that
	$R_\mu(P_\mu-(+\Id-\Id)-P_\nu)R_\nu=R_\mu(P_\mu-\Id)R_\nu+R_\mu(\Id-P_\nu)R_\nu=(\mu-\Id)R_\nu+R_\mu = (\mu-\nu)R_\nu - (R_\nu-R_\mu)$, where in the last line we used $\nu R_\nu = 0$. Hence, \eqref{avar:eq:PoiResolventkernelID} follows.
	
	For brevity, we will denote $P_n=P_{\mu_n}, R_{\mu_n}=R, P=P_\mu, R = R_\mu$ in what follows.
	
	\begin{lemma} \label{app:lemma:resolventconv}
		Assume Assumptions \ref{avar:hp:unifergodicity}, \ref{avar:hp:Vgrowth}, \ref{avar:hp:ptwiseapprox} and \ref{avar:hp:boundedder} hold. Then, for all $f\in\CbVfunX$, $R_nf(x)\rightarrow R_{\mu}f(x)$ for all $x\in\Xspace$ in probability.
	\end{lemma}
	\begin{proof}
		From \eqref{avar:eq:PoiResolventkernelID} we have that $R_nf-Rf = R(P_n-P)R_nf + (\mu-\mu_n)R_nf$. Since $R_nf$ has uniformly bounded $V$-norm by \cref{app:lemma:resolventbound}, the second term vanishes with $n\rightarrow\infty$ under \cref{avar:hp:ptwiseapprox}, and we just need to
		show that $R(P_n-P)R_nf(x)\rightarrow 0$ for all $x\in\Xspace$ \Pas. Since $(P_n-P)R_nf\in\BVfunX$ by \cref{app:lemma:resolventbound}, the series
		\begin{equation*}
			R(P_n-P)R_nf(x) = \sum_{k=1}^\infty[P^k(x,(P_n-P)R_nf)-\mu((P_n-P)R_nf)]
		\end{equation*}
		must be convergent, again by \cref{app:lemma:resolventbound}. In particular for any given $\varepsilon>0$ there exists a finite $\bar{k}$ independent of $n$ such that
		\begin{equation*}
			R(P_n-P)R_nf(x) = \sum_{k=\bar{k}}^\infty[P^k(x,(P_n-P)R_nf)-\mu((P_n-P)R_nf)] < \varepsilon.
		\end{equation*}   
		Therefore, we just need to show that
		\begin{equation} \label{avar:eq:lemmaproofconvk}
			P^k(x,(P_n-P)R_nf)-\mu((P_n-P)R_nf) \rightarrow 0
  		\end{equation}
		for $k=1,...\bar{k}$. Since $\bar{k}$ is finite and also independent of $n$, it is enough to show that \eqref{avar:eq:lemmaproofconvk} holds for a fixed finite $k$. In fact, by the derivative in the invariant measure being bounded by \cref{avar:hp:boundedder}, by mean value, for all $x\in\Xspace$ we have some finite positive $M_x,M_{\perp,x}$ such that 
		\begin{equation*}
			|(P-P_n)(x,R_nf)|\leq \Vnormf{R_nf}[M_x\Vnorm{\mu_n-\mu}+M_{\perp,x}|\mu_n(x)-\mu(x)|]
		\end{equation*}
		and the right hand side goes to 0 \Pas by \cref{app:lemma:resolventbound,avar:hp:ptwiseapprox,avar:hp:boundedder} for all $x\in\Xspace$. Also, $|(P-P_n)(x,R_nf)|$ is uniformly bounded in $n$ by a function in $\BVfunX$ by \cref{app:lemma:resolventbound}. Because $P^k(x,V)<\infty$ (\cref{avar:lemma:geomerg}), an application of the Dominated Convergence Theorem shows that \eqref{avar:eq:lemmaproofconvk} indeed holds \Pas for all $x\in\Xspace$. This proves  $R_nf(x)\rightarrow Rf(x)$ \Pas for all $x\in\Xspace$. 

        We can now use a standard separability argument to prove that $R_nf(x)\rightarrow Rf(x)$ for all $x\in\Xspace$ \Pas. Let $\{x_j\}$ be a countable dense subset of $\Xspace$, its existence guaranteed by $\Xspace$ being separable. Notice that then it is enough for the claim to prove that $R_nf(x)\rightarrow Rf(x)$ for all $x\in\Xspace$ if and only if $R_nf(x_j)\rightarrow Rf(x_j)$ for all $x_j\in\{x_j\}$ as the intersection of countably many probability 1 events has probability 1.
        The only if direction is trivial. To prove the if direction, we can write, for some $j$,
        \begin{equation*}
            |R_nf(x)-Rf(x)|\leq |R_nf(x)-R_nf(x_j)|+|Rf(x)-Rf(x_j)| + |R_nf(x_j) - Rf(x_j)|.
        \end{equation*}
        The last term is small with $n$ by hypothesis. Since $\{x_j\}$ is dense in $\Xspace$, we can pick a $x_j$ sufficiently close to $x$ so that the continuity of $x\mapsto R_nf(x)$ and $x\mapsto Rf(x)$ (which follows from \cref{hp:Fellertypechain} and the definition of the Poisson resolvent) ensure the first two terms are small, too. The claim follows. 
	\end{proof}
	
	For a kernel $K$ on $(\Xspace,\BorelX)$ define $\Vnorm{K}:=\sup_x \Vnorm{K(x,\cdot)}$. The next lemma is a more quantitative version of the one just proved.
	
	\begin{lemma} \label{app:lemma:resolventconvfast}
		Assume that Assumptions \ref{avar:hp:unifergodicity}, \ref{avar:hp:Vgrowth}, \ref{avar:hp:unifapprox} and \ref{avar:hp:unifboundedder} hold. Then, for $\alpha\in(0,1)$, we have $n^{-1/2}\sum_{i=1}^n \Vnorma{R_i-R_{i-1}}\rightarrow 0$ \Pas.
	\end{lemma}
	\begin{proof}
		Using \cite[Lemma A.1]{fort2014clt} and the hypothesis of uniformly bounded derivative in the invariant distribution \cref{avar:hp:unifboundedder} we have
		\begin{align*}
			\Vnorma{R_{i}-R_{i-1}} &\leq 3(L_{\mu_i}\wedge L_{\mu_{i-1}})^6 \mu_{i}(V^\alpha)\sup_x \Vnorma{P_{i}(x,\cdot)-P_{i-1}(x,\cdot)} \\
			&\leq \sup_i \mu_{i}(V^\alpha) \sup_i L_{\mu_i}^6[M_1\Vnorma{\mu_{i}-\mu_{i-1}}+M_2\sup_x|\mu_i(x)-\mu_{i-1}(x)|]
		\end{align*}
		and the claim follows by \cref{avar:lemma:geomerg,avar:hp:unifapprox}.
	\end{proof}

	\subsection{Central Limit Theorem for sMCMC} \label{app:smcmcproof} 
	We come to the proof of \cref{avar:thm:YCLT}. Let $f\in\CbVfunX$ and denote
	\begin{equation*}
		S_n(f) := n^{-1/2}\sum_{i=0}^n f(Y_{n,i})-\mu(f) = S_{1,n}(f) + S_{2,n}(f) 
	\end{equation*}
	where $S_{1,n}(f)= n^{-1/2}\left( \sum_{i=0}^n f(Y_{n,i})-\mu_n(f)\right)$ and $S_{2,n}(f)=n^{-1/2} (\mu_n-\mu)(f)$.
	\begin{proof}[Proof of \cref{avar:thm:YCLT}]
		We first show that $\mathbb{E}[\exp{(iuS_{1,n})}|\mathcal{F}_{n,0}]\rightarrow \exp(-(u^2/2)\sigma^2)$, $u\in\mathbb{R}^d$. Similarly to \cref{app:thm:Yllns}, we write $n^{-1/2} \sum_{i=1}^n f(Y_{n,i}) - \mu_n(f) = \sum_{i=1}^n U_{n,i} + r_n$, where $U_{n,i}:=n^{-1/2}[R_n f(Y_{n,i}) - P_n R_n f(Y_{n,i-1})]$, $r_{n}:=n^{-1/2}[P_nR_nf(Y_{n,0})-P_nR_nf(Y_{n,n})]$.
		
		$r_n\rightarrow_{\mathbb{P}}0$ follows from the same arguments used in the proof of \cref{app:thm:Yllns}. To verify the claim it then suffices to check $\mathbb{E}[\exp{(iu\sum_{i=1}^n U_{n,i})}|\mathcal{F}_{n,0}]\rightarrow \exp(-(u^2/2)\sigma^2)$, and we aim to apply \cite[Theorem A.3]{douc2007limit}. We need to check that
		\begin{equation*}
			\sum_{i=1}^N \mathbb{E}[U_{n,i}^2|\mathcal{F}_{n,i-1}] \rightarrow_{\mathbb{P}} \sigma^2 \quad \text{and} \quad \sum_{i=1}^n \mathbb{E}[|U_{n,i}|^21_{|U_{n,i}|\geq \epsilon}|\mathcal{F}_{n,i-1}]\rightarrow_{\mathbb{P}} 0.
		\end{equation*}
		The second condition follows using basically the same techniques we used in the proof of \cref{app:thm:Yllns} to check $\mathbb{E}[|U_{n,i}|1_{|U_{n,i}|\geq \epsilon}|\mathcal{F}_{n,i-1}]\rightarrow_{\mathbb{P}} 0$. To prove the first, notice that $\sum_{i=1}^N \mathbb{E}[U_{n,i}^2|\mathcal{F}_{n,i-1}]=n^{-1}\sum_{i=1}^N F_n(Y_{n,i})$,  where $F_{n}(y):=P_{n}(R_{n}f(y))^2-(P_{n}R_{n}f(y))^2$ and that $\sigma^2=\mu(F)$, with $F(y):=P(Rf(y)^2)-(PRf(y))^2$. We prove this via the LLN \cref{app:thm:Yllns}.
		Since by \cref{app:lemma:resolventbound} $\sup_n \Vnormfa{F_n}<\infty$, we only need to check that $F_n\rightarrow F$ pointwise in probability for all $f\in\BVfunX$. By the triangle inequality,
		\begin{align} \label{app:eq:Fnbound1}
			|F_n(y)-F(y)|&\leq |P_{n}(R_{n}f(y)^2)-P(R_{n}f(y)^2)|+|(P_{n}R_{n}f(y))^2-(PR_{n}f(y))^2| \nonumber \\ 
			&+ |(PR_{n}f(y))^2-(PRf(y))^2|+ |P(R_{n}f(y)^2)-P(Rf(y)^2)|,
		\end{align}
		and we can check that, indeed, since the derivative in the invariant distribution of $P$ is bounded at $\mu$ towards every $\mu_n$,
		\begin{align}  \label{app:eq:Fnbound2}
			&|P_{n}(R_{n}f(y)^2)-P(R_{n}f(y)^2)|  
			\leq \Vnormfa{R_{n}f^2}[M_{1,y}\Vnorma{\mu_n-\mu}+M_{2,y}|\mu_n(y)-\mu(y)|].
		\end{align}
		The right hand side goes to zero under our assumptions by \cref{app:lemma:resolventbound}.
		Similarly, the second term of the triangle inequality vanishes because
		\begin{align}  \label{app:eq:Fnbound3}
			&|(P_{n}R_{n}f(y))^2-(PR_{n}f(y))^2|   
			=|P_{n}R_{n}f(y)+PR_{n}f(y)||P_{n}R_{n}f(y)-PR_{n}f(y)|  \nonumber \\
			&\leq |P_{n}R_{n}f(y)+PR_{n}f(y)|\Vnormfa{R_{n}f}[M_{1,y}\Vnorma{\mu_n-\mu}+M_{2,y}|\mu_n(y)-\mu(y)|] \nonumber \\
			&\leq 2\Vnormfa{R_{n}f}^2[M_{1,y}\Vnorma{\mu_n-\mu}+M_{2,y}|\mu_n(y)-\mu(y)|]
		\end{align}
		Also, since $|P(R_{n}f(y)^2)-P(Rf(y)^2)| = |P(y,R_{n}f^2-R f^2)|$, the third and the fourth vanish by Bounded Convergence Theorem and \cref{app:lemma:resolventconv,app:lemma:resolventbound}. Thus, we have proved
		\begin{equation*}
			\mathbb{E}(\exp (iuS_{1,n}(f))|\mathcal{F}_{n,0})\rightarrow_{\mathbb{P}} \exp \left(-u^2\sigma^2(f)/2\right), \quad u\in\mathbb{R}^d.
		\end{equation*}
		Furthermore, by \cref{avar:hp:Yapproxnormality}, for $f\in\mathcal{G}(\Xspace)$,
		\begin{equation*}
			\mathbb{E}(\exp (iuS_{2,n}(f)))\rightarrow \exp \left(-u^2v^2(f)/2\right), \quad u\in\mathbb{R}^d
		\end{equation*}
		for $f\in\mathcal{G}(\Xspace)$. It follows that for $f\in \CbVfunX\cap \mathcal{G}(\Xspace)$,
		\begin{align*}
			\mathbb{E} (\exp(iuS_n(f))) &= \mathbb{E}((\mathbb{E}(\exp (iuS_{1,n}(f))|\mathcal{F}_{n,0})-\exp(-u^2\sigma^2(f)/2))\exp(iuS_{2,n}(f)))) \\
			&+\exp(-u^2\sigma^2(f)/2)  \mathbb{E}(\exp (iuS_{2,n}(f))) \rightarrow_{\mathbb{P}} \exp (-u^2(v^2(f)+\sigma^2(f))/2).
		\end{align*}  
	\end{proof}
	
	\subsection{Central Limit Theorem for iMCMC} \label{app:imcmcproof}
	
	We come to the proof of \cref{avar:thm:ZCLT}. Let $f\in\CbVfunXa$ and denote
	\begin{equation*}
		S_n(f) := n^{-1/2}\sum_{i=1}^n f(Z_i)-\mu(f) = S_{1,n}(f) + S_{2,n}(f) 
	\end{equation*}
	where $S_{1,n}(f)= n^{-1/2}\sum_{i=1}^n f(Y_i)-\mu_{i-1}(f)$ and $S_{2,n}(f)=n^{-1/2}\sum_{i=1}^n(\mu_{i-1}-\mu)(f)$. 
	\begin{proof}[Proof of \cref{avar:thm:ZCLT}]
		We first show that $\mathbb{E}[\exp{(iuS_{1,n})}|\mathcal{F}_{n,0}]\rightarrow \exp(-(u^2/2)\sigma^2)$, $u\in\mathbb{R}^d$. To do that, we check the conditions of \cite[Theorem 2.2]{fort2014clt} to prove that $S_{1,n}\Rightarrow N(0,\sigma^2(f))$. Their A1 and A2 are satisfied under our assumption by \cref{avar:hp:unifergodicity}. To check their A3 we need to check that
		\begin{equation} \label{app:eq:ZverifyA3}
			\frac{1}{\sqrt{n}} \sum_{i=1}^n \Vnormfa{P_{i}R_{i}f-P_{i-1}R_{i-1}f}V^\alpha(Z_i) \rightarrow_{\mathbb{P}} 0
		\end{equation}
		and
		\begin{equation*}
			\frac{1}{n^{1/2\alpha}} \sum_{i=1}^n L_{i}^{2/\alpha}P_{i}V(Z_i) \rightarrow_{\mathbb{P}} 0.
		\end{equation*}
		The last condition is satisfied by \cref{avar:lemma:geomerg} since $\sup_n V(Z_n)$ is \Pas-finite and $\sup_n L_n<\infty$. To verify the first, we use \cite[Lemma A.1]{fort2014clt} and then the uniformly bounded derivative \cref{avar:hp:unifboundedder} with a mean value inequality to write
		\begin{align*}
			\Vnormfa{P_{i}R_{i}f-P_{i-1}R_{i-1}f} 
			&\leq \Vnormfa{f}5L^6\mu_{i}(V^\alpha)\sup_x \Vnorma{P_{i}(x,\cdot)-P_{{i-1}}(x,\cdot)} \\
			&\leq \Vnormfa{f}5L^6\mu_{i}(V^\alpha)[M_1\Vnorma{\mu_i-\mu_{i-1}}+M_2\sup_x|\mu_i(x)-\mu_{i-1}(x)|]
		\end{align*}
		for some $M_1,M_2<\infty$. Since \cref{avar:lemma:geomerg} shows that $\limsup_n \mu_n(V)<\infty$, under \cref{avar:hp:unifapprox}, \eqref{app:eq:ZverifyA3} is verified. We now need to check A4 in  \cite[Theorem 2.2]{fort2014clt}, namely, that $n^{-1}\sum_{i=1}^n F_{i}(Z_i)\rightarrow_{\mathbb{P}}\sigma^2(f)=\mu(F)$, where $F_{n}(y):=P_{n}(R_{n}f(y))^2-(P_{n}R_{n}f(y))^2$ and $F(y):=P(R f(y)^2)-(PRf(y))^2$. For this we use \cref{app:thm:Zllns}. By \cref{app:lemma:resolventbound}, $\sup_n \Vnormfa{F_n}<\infty$, and by using the same arguments of the proof of \cref{avar:thm:YCLT}, we see that $F_n(y)\rightarrow F(y)$ for all $f\in\CbVfunXa$. Hence we only need to verify that $ \sum_{i=1}^\infty i^{-1}\sup_x \Vnormf{F(\cdot,\mu_i)-F(\cdot,\mu_{i-1})}<\infty$.   
		Since 
		\begin{align*}
			|(P_{{i-1}}R_if(x))^2&-(P_{i-1}R_{i-1}f(x))^2+ |P_{i-1}(R_if(y)^2)-P_{i-1}(R_{i-1}f(y)^2)|  \\ 
			&\leq \sup_x|R_if(x)+R_{i-1}f(x)||R_if(x)-R_{i-1}f(x)|+ \sup_x |R_if^2(x)-R_{i-1}f^2(x)| \\
			&\leq \Vnorma{R_i-R_{i-1}}c_i
		\end{align*}
		with $c_i:=[|f^2|_{V^\alpha}+|f|_{V^\alpha}\sup_x|R_if(x)+R_{i-1}f(x)|]$, using again the mean value estimates \eqref{app:eq:Fnbound1}-\eqref{app:eq:Fnbound3} with $\mu_{i-1}$ instead of $\mu$ and $M_1, M_2$ instead of $M_{1,y}, M_{2,y}$ (in light of \cref{avar:hp:unifboundedder}) we obtain
		\begin{align*}
			\sum_{i=1}^\infty i^{-1}\sup_x \Vnormf{F(\cdot,\mu_i)-F(\cdot,\mu_{i-1})}
			\leq 
			\sum_{i=1}^\infty &i^{-1}[M_{1}\Vnorma{\mu_i-\mu_{i-1}}+M_{2}\sup_x|\mu_i(x)-\mu_{i-1}(x)|]d_i \\ &+\Vnorma{R_i-R_{i-1}}c_i
		\end{align*}
		where $d_i:=2|R_if^2|_{V^\alpha}+ 2|R_if|^2_{V^\alpha}$. Since $\sup_n c_n,d_n<\infty$ by \cref{avar:hp:unifapprox,app:lemma:resolventconvfast} the condition is verified. This proves that $S_{1,n}\Rightarrow N(0,\sigma^2(f))$ and then \eqref{avar:eq:ZrandomcenterCLT}.
		The proof for \eqref{avar:eq:ZdetcenterCLT} from \eqref{avar:eq:ZrandomcenterCLT} follows the same steps as the one for \eqref{avar:eq:YdetcenterCLT} from \eqref{avar:eq:YrandomcenterCLT}.
	\end{proof}

\end{document}